\documentclass[preprint,12pt,3p]{elsarticle}
\usepackage{verbatim}
\usepackage{amsmath,amsthm,amsfonts,amssymb}
\usepackage{algorithm}
\usepackage{algpseudocode}
\usepackage{cases}
\usepackage{hyperref}
\usepackage{multirow}
\usepackage[active,new,old]{correct} 
\usepackage{amsmath,amsfonts,amssymb}
\usepackage[normalem]{ulem}
\usepackage{xcolor,sectsty}
\usepackage{natbib}
\definecolor{astral}{RGB}{46,116,181}
\subsectionfont{\color{astral}}
\sectionfont{\color{astral}}
\newtheorem{theorem}{Theorem}[section]
\newtheorem{lemma}[theorem]{Lemma}

\newtheorem{definition}[theorem]{Definition}
\newtheorem{example}[theorem]{Example}

\newtheorem{remark}[theorem]{Remark}
\newtheorem{thm}{Theorem}[section]

\newtheorem{corollary}[thm]{Corollary}
\newtheorem{preposition}[thm]{Preposition}

\definecolor{darkslategray}{rgb}{0.18, 0.31, 0.31}
\definecolor{warmblack}{rgb}{0.0, 0.26, 0.26}
\journal{arxiv}

\newcommand{\mc}[1]{\mathcal {#1}}

\newcommand{\n}{{*_N}}

\newcommand{\2}{{*_2}}
\newcommand{\m}{{*_M}}
\newcommand{\kl}{{*_L}}

\begin{document}

\begin{frontmatter}

\title{ \textcolor{warmblack}{\bf Generalized Inverses of Boolean Tensors via Einstein Product}}

\author{Ratikanta Behera $^*$ and Jajati Keshari Sahoo $^\dag$ }

\address{$^{*}$
Department of Mathematics and Statistics,\\
Indian Institute of Science Education and Research Kolkata,\\
 Nadia, West Bengal, India.\\
\textit{E-mail}: \texttt{ratikanta@iiserkol.ac.in}

                \vspace{.3cm}

                       $^{\dag}$ Department of Mathematics,\\
                       BITS Pilani, K.K. Birla Goa Campus, Goa, India
                        \\\textit{E-mail}: \texttt{jksahoo\symbol{'100}goa.bits-pilani.ac.in }
                       }

\begin{abstract}
\textcolor{warmblack}{
Applications of the theory and computations of Boolean matrices are of fundamental importance to study a variety of discrete structural models.  But the increasing ability of data collection systems to store huge volumes of multidimensional data, the Boolean matrix representation of data analysis is not enough to represent all the information content of the multiway data in different fields.  From this perspective, it is appropriate to develop an infrastructure that supports reasoning about the theory and computations. 
In this paper, we discuss the generalized inverses of the Boolean tensors with the Einstein product. Further, we elaborate on this theory by producing a few characterizations of different generalized inverses and several equivalence results on Boolean tensors. We explore the space decomposition of the Boolean tensors and present reflexive generalized inverses through it. In addition to this, we address rank and the weight for the Boolean tensor.
}
\end{abstract}


\end{frontmatter}

\section{Introduction}
\label{sec1}

\subsection{Background and motivation}

The study of the Boolean matrices \cite{bapat, bapatbook, luce, rao} play an important role in linear algebra \cite{belohlavek2015below, kim, rao72}, combinatorics \cite{brual}, graph theory \cite{berge}  and network theory \cite{ledley, li2015logical}. However, this becomes particularly challenging to store huge volumes of multidimensional data.  This potential difficulty can be easily overcome, thanks to tensors,  which are natural multidimensional generalizations of matrices \cite{kolda, loan}. Here the notion of tensors is different in physics and engineering (such as stress tensors) \cite{Nar93}, which are generally referred to as tensor fields in mathematics \cite{de2008tensor}.  However, it will be more appropriate if we study the Boolean tensors and the generalized inverses of Boolean tensors. Hence the generalized inverses of Boolean tensors will encounter in many branches of mathematics, including relations theory \cite{plem}, logic, graph theory, lattice theory \cite{birkhoff} and algebraic semigroup theory.

Recently, there has been increasing interest in studying inverses \cite{BraliNT13} and different generalized inverses of tensors based on the Einstein product \cite{bm, wei18,  stan, sun}, and opened new perspectives for solving multilinear systems \cite{kolda, Mao19}. In \cite{wei18, stan}, the authors have introduced some basic properties of the range and null space of multidimensional arrays. Further, in \cite{stan}, it was discussed the adequate definition of the tensor rank, termed as reshaping rank. Corresponding representations of the weighted Moore-Penrose inverse introduced in \cite{BehMM19, we17} and investigated a few characterizations in \cite{PanMi19}. Though this work is focusing on the binary case; i.e., concentrating some interesting results based on the Boolean tensors and generalized inverses of Boolean tensors via the Einstein product.  In many instances, the result in the general case does not immediately follow even though it is not difficult to conclude.

On the other hand, one of the most successful developments in the world of multilinear algebra is the concept of tensor decomposition \cite{ Kolda01, kolda, LalmV00}. This concept gives a clear and convenient way to implement all basic operations efficiently.  Recently this concept is extended in Boolean tensors \cite{, erdos2013discovering, khamis2017Boolean, rukat2018tensormachine}. Further,  the fast and scalable distributed algorithms for Boolean tensor decompositions were discussed in  \cite{miettinen2011Boolean}. In addition to that, a few applications of these decompositions are discussed in \cite{erdos2013discovering, metzler2015clustering} for information extraction and clustering. At that same time, Brazell, et al. in \cite{BraliNT13} discussed decomposition of tensors from the isomorphic group structure on the influence of the Einstein Product and demonstrated that they are special cases of the canonical polyadic decomposition \cite{carroll1970analysis}. The vast work on decomposition on the tensors and its several applications in different areas of mathematics in the literature, and the recent works in \cite{BraliNT13, sun}, motivate us to study the generalized inverses and space decomposition in the framework of Boolean tensors. 
 This study leads to introduce the rank and the weight for the Boolean tensor with its application to  generalized inverses.

\subsection{Organization of the paper}

The rest of the paper is organized as follows. In Section 2 we present some definitions, notations, and preliminary results, which are essential in proving the main results.  The main results are discussed in Section 3. It has four subparts. In the first part, some identities are proved while the generalized inverses for Boolean tensor are discussed in the second part. The third part mainly focuses on weighted Moore-Penrose inverses. Space decomposition and its application to generalized inverses are discussed in the last part.  Finally, the results along with a few questions are concluded in Section 4.

\section{Preliminaries}
We first introduce some basic definitions and notations which will be used throughout the article. 
\subsection{Definitions and terminology}
For convenience, we first briefly explain some of the terminologies which will be used here onwards. The tensor notation and definitions are followed from the article \cite{BraliNT13, sun}.  We refer  $\mathbb{R}^{I_1\times\cdots\times I_N}$  as the set of order $N$  real tensors. Indeed, a matrix is a second order tensor, and a vector is a first order tensor. Let $\mathbb{R}^{I_1\times\cdots\times I_N}$  be the set of  order $N$ and dimension $I_1 \times \cdots \times I_N$ tensors over the real
field $\mathbb{R}$.  $\mc{A} \in \mathbb{R}^{I_1\times\cdots\times I_N}$ is a tensor with  $N$-th order tensor, and each entry of $\mc{A}$ is denoted by $a_{i_1...i_N}$.  Note that throughout the paper, tensors are represented in calligraphic letters like  $\mc{A}$, and the notation $(\mc{A})_{i_1...i_N}= a_{i_1...i_N}$ represents the scalars. The Einstein product (\cite{ein}) $ \mc{A}\n\mc{B} \in \mathbb{R}^{I_1\times\cdots\times
I_N \times J_1 \times\cdots\times J_M }$ of tensors $\mc{A} \in \mathbb{R}^{I_1\times\cdots\times I_N \times K_1
\times\cdots\times K_N }$ and $\mc{B} \in
\mathbb{R}^{K_1\times\cdots\times K_N \times J_1 \times\cdots\times
J_M }$   is defined
by the operation $\n$ via
\begin{equation}\label{Eins}
(\mc{A}\n\mc{B})_{i_1...i_Nj_1...j_M}
=\displaystyle\sum_{k_1...k_N}a_{{i_1...i_N}{k_1...k_N}}b_{{k_1...k_N}{j_1...j_M}}.
\end{equation}
Specifically, if $\mc{B} \in \mathbb{R}^{K_1\times\cdots\times K_N}$, then $\mc{A}\n\mc{B} \in \mathbb{R}^{I_1\times\cdots\times I_N}$ and 
\begin{equation*}\label{Einsb}
(\mc{A}\n\mc{B})_{i_1...i_N} = \displaystyle\sum_{k_1...k_N}
a_{{i_1...i_N}{k_1...k_N}}b_{{k_1...k_N}}.
\end{equation*}
This product is discussed in the area of continuum mechanics  \cite{ein} and the theory of relativity \cite{lai}. Further, the addition of two tensors $\mc{A}, ~\mc{B}\in \mathbb{R}^{I_1\times\cdots\times I_N \times K_1 \times\cdots\times K_N }$ is defined
as 
\begin{equation}\label{Eins1}
(\mc{A} + \mc{B})_{i_1...i_N k_1...k_N}
=a_{{i_1...i_N}{k_1...k_N}} + b_{{i_1...i_N}{k_1...k_N}}.
\end{equation}
For a tensor $~\mc{A}=(a_{{i_1}...{i_N}{j_1}...{j_M}})
 \in \mathbb{R}^{I_1\times\cdots\times I_N \times J_1 \times\cdots\times J_M},$  let  $\mc{B} =(b_{{i_1}...{i_M}{j_1}...{j_N}}) \in \mathbb{R}^{J_1\times\cdots\times J_M \times I_1 \times\cdots\times I_N}$, be the {\it transpose} of $\mc{A}$, where $b_{i_1\cdots i_Mj_1\cdots j_N} = a_{j_1\cdots j_M i_1\cdots i_N}.$ The tensor $\mc{B}$ is denoted by $\mc{A}^T$. Also, we denote  $\mc{A}^T=\left({a}_{{i_1}...{i_N}{j_1}...{j_M}}^t\right).$ 
 The trace of a tensor $\mc{A}$ with entries $(\mc{A})_{{i_1}...{i_N}{j_1}...{j_N}}$, denoted by $tr(\mc{A})$,
  is defined as the sum of the diagonal entries, i.e., 
$tr(\mc{A}) = \displaystyle\sum_{i_1 \cdots i_N}a_{{i_1...i_N}{i_1...i_N}}.$ Further, a tensor $\mc{O}$ denotes the {\it zero tensor} if  all the entries are zero. 
 A tensor $\mc{A}\in
\mathbb{R}^{I_1\times\cdots\times I_N \times I_1 \times\cdots\times
I_N}$ is {\it symmetric}  if  $\mc{A}=\mc{A}^T,$  and {\it orthogonal} if $\mc{A}\m\mc{A}^T= \mc{A}^T\n \mc{A}=\mc{I}$. Further, a tensor
$\mc{A}\in \mathbb{R}^{I_1\times\cdots\times I_N \times I_1
\times\cdots\times I_N}$  is {\it idempotent}  if $\mc{A}
\n \mc{A}= \mc{A}$.  The definition of a diagonal tensor follows. 
Further, a tensor with entries $(\mc{D})_{{i_1}...{i_N}{j_1}...{j_N}}$ is
   called a {\it diagonal
   tensor} if $d_{{i_1}...{i_N}{j_1}...{j_N}} = 0$ for $(i_1,\cdots,i_N) \neq (j_1,\cdots,j_N).$
A few more notations and definitions are discussed below for defining generalized inverses of Boolean tensors. We first recall the definition
of an identity  tensor below.

\begin{definition} (Definition 3.13, \cite{BraliNT13}) \\
A tensor 
  with entries 
     $ (\mc{I})_{i_1 \cdots i_Nj_1\cdots j_N} = \prod_{k=1}^{N} \delta_{i_k j_k}$,
   where
\begin{numcases}
{\delta_{i_kj_k}=}
  1, &  $i_k = j_k$,\nonumber
  \\
  0, & $i_k \neq j_k $.\nonumber
\end{numcases}
 is  called a {\it  unit tensor or identity tensor}.
\end{definition}
The permutation tensor is defined as follows.

\begin{definition}\label{perm} 
Let $\pi$ be a permutation map  on $(i_1,i_2,\cdots, i_N,j_1,j_2,\cdots , j_N)$ defined by 
$$\pi:=\begin{pmatrix}
i_1&i_2&\cdots &i_N&j_1&j_2&\cdots &j_N \\
\pi(i_1)&\pi(i_2)&\cdots &\pi(i_N)&\pi(j_1)&\pi(j_2)&\cdots&\pi(j_N)  \\
\end{pmatrix}.
$$
A  tensor $\mc{P}$
  with entries 
     $ (\mc{P})_{i_1 \cdots i_Nj_1\cdots j_N} = \prod_{k=1}^{N} \epsilon_{i_k} \epsilon_{j_k}$,
   where
\begin{numcases}
{\epsilon_{i_k}\epsilon_{j_k} = }
  1, &  $\pi(i_k) = j_k$,\nonumber
  \\
  0, & otherwise.\nonumber
\end{numcases}
 is  called a {\it  permutation tensor}.
\end{definition}

Now we recall the block tensor as follows. 
\begin{definition}{\cite{sun}}
For a tensor $\mc{A} = (a_{i_1... i_Nj_1...j_M})
 \in \mathbb{R}^{I_1\times\cdots\times I_N \times J_1 \times\cdots\times
 J_M},\\
  \mc{A}_{(i_1...i_N|:)}= (a_{i_1...i_N:...:})\in \mathbb{R}^{J_1\times\cdots\times J_M}$ is a
  subblock of $\mc{A}$. $Vec(\mc{A})$ is obtained by lining up all the subtensors
  in a column, and  $t$-th subblock of $Vec(\mc{A})$ is  $\mc{A}_{(i_1...i_N|:)}$,
 where $$t=i_N + \displaystyle\sum_{K=1}^{N-1} \left[ (i_K - 1) \displaystyle\prod_{L=K+1}^{N} I_L \right].$$
\end{definition}
\vspace{-0.5cm}
Let $\mc{A} = (a_{i_1\cdots i_N j_1 \cdots j_M}) \in
\mathbb{R}^{I_1\times\cdots\times I_N \times J_1 \times\cdots\times
J_M}$ and $\mc{B} = (b_{i_1\cdots i_N k_1 \cdots k_M}) \in
\mathbb{R}^{I_1\times\cdots\times I_N \times K_1 \times\cdots\times
K_M}$. The {\it row block tensor} consisting  of $\mc{A}$ and
$\mc{B}$ is denoted by
$[\mc{A} ~ \mc{B}] \in \mathbb{R}^{\alpha^N\times\beta_1\times \cdots \times \beta_M},$
where $\alpha^N = I_1\times\cdots\times I_N, \beta_i = J_i + K_i, i
= 1, \cdots, M$, and is defined by
\begin{equation*}
[\mc{A} ~ \mc{B}]_{i_1 \cdots i_N l_1 \cdots l_M} =
\begin{cases}
a_{i_1 \cdots i_N l_1 \cdots l_M}, & i_1 \cdots i_N \in [I_1] \times \dots \times [I_N], l_1 \cdots l_M \in [J_1] \times \cdots \times [J_M];
\\
b_{i_1 \cdots i_N l_1 \cdots l_M}, & i_1 \cdots i_N \in [I_1] \times \dots \times [I_N], l_1 \cdots l_M \in \Gamma_1 \times \cdots \times \Gamma_M;
\\
0, & \textnormal{otherwise}.
\end{cases}
\end{equation*}
where $\Gamma_i = \{ J_i +1, \cdots, J_i+K_i\}, i=1,\cdots, M.$

Let $\mc{C} = (c_{j_1 \cdots j_M i_1 \cdots i_N}) \in
\mathbb{R}^{J_1\times\cdots\times J_M \times I_1 \times\cdots\times
I_N}$ and $\mc{D} = (d_{k_1 \cdots k_M i_1 \cdots i_N})  \in
\mathbb{R}^{K_1\times\cdots\times K_M \times I_1 \times\cdots\times
I_N}$. The {\it column block tensor} consisting of $\mc{C}$ and
$\mc{D}$ is
\begin{equation*}\label{eq224}
\left[%
\begin{array}{c}
  \mc{C} \\
  \mc{D} \\
\end{array}%
\right]= [\mc{C}^T ~ \mc{D}^T]^T \in \mathbb{R}^{\beta_1 \times
\cdots \times \beta_M\times\alpha^N}.
\end{equation*}
For $\mc{A}_1 \in \mathbb{R}^{I_1\times\cdots\times I_N \times J_1
\times\cdots\times J_M}, \mc{B}_1 \in
\mathbb{R}^{I_1\times\cdots\times I_N \times K_1 \times\cdots\times
K_M}, \mc{A}_2 \in \mathbb{R}^{L_1\times\cdots\times L_N \times J_1
\times\cdots\times J_M}$ and $ \mc{B}_2 \in
\mathbb{R}^{L_1\times\cdots\times L_N \times K_1 \times\cdots\times
K_M}$, we denote $\tau_1 = [\mc{A}_1 ~ \mc{B}_1]$ and $\tau_2 =
[\mc{A}_2 ~ \mc{B}_2]$ as the {\it row block tensors}.
 The {\it column block tensor}  $ \left[
\begin{array}{c}
  {\tau}_1 \\
  {\tau}_2 \\
\end{array}
\right]
$ can be written as
\begin{equation*}\label{eq225}
\left[
\begin{array}{c}
  \mc{A}_1 ~~ \mc{B}_1\\
  \mc{A}_2 ~~ \mc{B}_2\\
\end{array}
\right] \in \mathbb{R}^{\rho_1\times\cdots\times \rho_N \times \beta_1 \times\cdots\times \beta_M},
\end{equation*}
where $\rho_i = I_i +L_i, i=1,\cdots,N; \beta_j = J_j + K_j$ and $j=1,\cdots , M.$

\begin{definition} (Definition 2.1, \cite{stan}) \\
The range space and null space of a tensor $\mc{A}\in \mathbb{R}^{{I_1}\times \cdots\times {I_M}\times {J_1}\times\cdots \times {J_N}}$ are defined as per the following: 
$$
\mathfrak{R}(\mc{A}) = \left\{\mc{A}\n\mc{X}:~\mc{X}\in\mathbb{R}^{{J_1}\times\cdots\times {J_N}}\right\}\mbox{ and } \mc{N}(\mc{A})=\left\{\mc{X}:~\mc{A}\n\mc{X}=\mc{O}\in\mathbb{R}^{{I_1}\times \cdots \times {I_M}}\right\}.
$$
\end{definition}
The relation of range space for tensors is discussed in \cite{stan} as follows.
\begin{lemma}[Lemma 2.2. \cite{stan}]\label{range-stan}
Let  $\mc{A}\in \mathbb{R}^{{I_1}\times\cdots\times {I_M}\times {J_1}\times\cdots\times {J_N}}$,  $\mc{B}\in \mathbb{R}^{{I_1}\times\cdots\times {I_M}\times {K_1}\times\cdots\times {K_L}}.$ Then $\mathfrak{R}(\mc{B})\subseteq\mathfrak{R}(\mc{A})$ if and only if there exists  $\mc{U}\in \mathbb{R}^{{J_1}\times\cdots\times {J_N}\times {K_1}\times\cdots\times {K_L}}$ such that 
$\mc{B}=\mc{A}\n\mc{U}.$
\end{lemma}

The next subsection is discussed the Boolean tensor and some useful definitions
\subsection{The Boolean tensor}
The binary Boolean algebra $\mathfrak{B}$ consists of the set $\{0,1\}$ equipped with the operations of addition and multiplication defined as follows:
\begin{center}
\begin{tabular}{ c|c c } 
 & 0 & 1 \\
\hline
 0 & 0 & 1 \\ 
1 & 1 & 1 \\ 
\end{tabular} 
\hspace{2cm}
\begin{tabular}{ c|c c } 
. & 0 & 1 \\
\hline
 0 & 0 & 0 \\ 
1 & 0 & 1 \\ 
\end{tabular}
\end{center}
\begin{definition}
Let $\mc{A}=(a_{{i_1}...{i_M}{j_1}...{j_N}})
 \in \mathbb{R}^{I_1\times\cdots\times I_M \times J_1 \times\cdots\times J_N}.$ If $a_{{i_1}...{i_M}{j_1}...{j_N}}\in\{0,1\},$ then the tensor $\mc{A}$ is called Boolean tensor. 
 \end{definition}
 The addition and product of Boolean tensors  are defined as in Eqs. (\ref{Eins}) and (\ref{Eins1}) but addition and product of two entries will follow addition and product rule of Boolean algebra.
The order relation for tensors is defined as follows.
\begin{definition}
Let $\mc{A}=(a_{{i_1}...{i_M}{j_1}...{j_N}})
 \in \mathbb{R}^{I_1\times\cdots\times I_M \times J_1 \times\cdots\times J_N} \text{ and }~\mc{B}
  =(b_{{i_1}...{i_M}{j_1}...{j_N}})~~ \in  \mathbb{R}^{I_1\times\cdots\times I_M
\times J_1 \times\cdots\times J_N}. $ Then $\mc{A}\leq \mc{B}$ if and only if $a_{{i_1}...{i_M}{j_1}...{j_N}}\leq b_{{i_1}...{i_M}{j_1}...{j_N}}$ for all $i_s$ and $j_t$ where $1\leq s\leq M$ and $1\leq t\leq N.$
\end{definition}

 We generalize the component-wise complement of the Boolean matrix \cite{fitz} to Boolean tensors and defined below.

\begin{definition}\label{CompDef} 
Let  $\mc{A}=(a_{{i_1}...{i_N}{j_1}...{j_M}})
 \in \mathbb{R}^{I_1\times\cdots\times I_N \times J_1 \times\cdots\times J_M}$ be a Boolean tensor. A tensor $\mc{B}=(b_{{i_1}...{i_N}{j_1}...{j_M}})
 \in \mathbb{R}^{I_1\times\cdots\times I_N \times J_1 \times\cdots\times J_M}$ is called component-wise complement of $\mc{A}$ if 
\begin{equation*}
b_{{i_1}...{i_N}{j_1}...{j_M}}=\left\{\begin{array}{cc}
   1,   & \mbox{  when  } a_{i_1\cdots i
  _Nj_1j_2\cdots j
  _M}=0. \\
   0,   & \mbox{  when } a_{i_1\cdots i
  _Nj_1j_2\cdots j
  _M}=1.
 \end{array}\right.
 \end{equation*}
The tensor $\mc{B}$ and its entries respectively, denoted by    $\mc{A}^C$ and  $\left(a_{i_1\cdots i
  _Nj_1\cdots j
  _M}^c\right).$
\end{definition}

 \section{Main Results} 
In this section, we prove a few exciting results on tensors which are emphasized in the binary case. We divided this section into four folds.  In the first part of this section, we discuss some identities on the Boolean tensors.  Then, after having introduced some necessary ingredients, we study the generalized inverses of the Boolean tensor and some equivalence results to other generalized inverses in the second part.  The existence and uniqueness of weighted Moore-Penrose inverses are discussed in the third part. The space decomposition and its connection to generalized inverses are presented in the final part.

\subsection{Some identities on Boolean tensors}
   
By the definition of  Boolean tensor $\mc{A}\in\mathbb{R}^{I_1 \times \cdots \times I_M\times I_1\times \cdots \times I_M},$ we always get $\mc{A}+\mc{A}=\mc{A}.$ The infinite series of the Boolean tensor,  $\displaystyle\sum_{k=1}^\infty\mc{A}^k$, is convergent and reduces to a finite series, since there are only finite number of Boolean tensors of the same order. Now we denote  $\overline{\mc{A}}$ for the infinite series of the Boolean tensors, i.e.,  $$\overline{\mc{A}}=\displaystyle\sum_{k=1}^ \infty \mc{A}^k.$$

Since $\mc{A}\leq\mc{A}+\mc{B}$ for any two Boolean tensor  (suitable order for addition)  $\mc{A}$ and $\mc{B}$, likewise $\mc{A}=\mc{A}+\mc{A}\geq \mc{A}+\mc{B}$ for any two Boolean tensor $\mc{A}\geq \mc{B}$. This is stated in the next result.

\begin{theorem}\label{thm3.11}
Let $\mc{A}\in\mathbb{R}^{I_1\times \cdots \times I_M\times J_1 \times\cdots \times J_N}$ and $\mc{B}\in\mathbb{R}^{I_1 \times \cdots \times I_M\times J_1\times\cdots \times J_N}.$ Then $\mc{A}\geq\mc{B}$ if and only if $\mc{A}+\mc{B}=\mc{A}$.
\end{theorem}

If we consider $\mc{A}\geq \mc{I}$ in the above theorem, then it is easy to verify that   
$\mc{I}+\mc{A}+\cdots+\mc{A}^n=\mc{A}^n$ and hence we can have the following result as a corollary. 
\begin{corollary}\label{cor3.4}
Let $\mc{A}\in\mathbb{R}^{I_1 \times\cdots \times I_N\times I_1 \times\cdots \times I_N}$ and $\overline{\mc{A}}=\sum_{k=1}^\infty \mc{A}^k.$ If $\mc{A}\geq\mc{I},$ then there exist $n,$ such that
\begin{enumerate}
    \item[(a)] $\overline{\mc{A}}=\mc{A}^{n};$
    \item[(b)] $\left(\overline{\mc{A}}\right)^2=\overline{\mc{A}};$
    \item[(c)] $\overline{\left(\overline{\mc{A}}\right)}=\overline{\mc{A}}.$
\end{enumerate}
 \end{corollary}

Using the above theorem, we now prove another result on the Boolean tensor. As follows,

\begin{theorem}
Let $\mc{A}\in\mathbb{R}^{I_1\times \cdots \times  I_N\times I_1 \times \cdots \times I_N}$ and $\mc{B}\in\mathbb{R}^{I_1 \times \cdots \times I_N\times I_1\times \cdots \times  I_N},$ with $\mc{A}\geq \mc{I}$ and $\mc{B}\geq \mc{I}.$ Then $$\overline{(\mc{A}+\mc{B})}=\overline{(\overline{\mc{A}}\n\overline{\mc{B}})}=\overline{(\overline{\mc{B}}\n\overline{\mc{A}})}.$$ 
\end{theorem}

\begin{proof}
Since $\mc{A}\geq \mc{I}$ and $\mc{B}\geq \mc{I}.$ So $\overline{\mc{A}}\geq \mc{I}$ and  $\overline{\mc{B}}\geq \mc{I}.$  Also we have $\overline{\mc{A}}\geq \mc{A}$ and $\overline{\mc{B}}\geq \mc{B}.$ Combining these results, we get $\overline{A}\n\overline{B}\geq \mc{A}$ and $\overline{A}\n\overline{B}\geq \mc{B}.$ Thus $\overline{A}\n\overline{B}\geq \mc{A}+\mc{B}$ and hence
\begin{equation}\label{eq3.61}
  \overline{\left(\overline{\mc{A}}\n\overline{\mc{B}}\right)}\geq \overline{\mc{A}+\mc{B}}.  
\end{equation}
Now $\overline{\mc{A}+\mc{B}}\geq \overline{\mc{A}}$ and $\overline{\mc{A}+\mc{B}}\geq \overline{\mc{B}}.$  By using Corollary \ref{cor3.4} $(c)$, we get $\overline{\mc{A}}\n\overline{\mc{B}}\leq \left(\overline{\mc{A}+\mc{B}}\right)^2=\overline{\mc{A}+\mc{B}}.$ From Corollary \ref{cor3.4} $(b)$, we have 
\begin{equation}\label{eq3.362}
    \overline{\left(\overline{\mc{A}}\n\overline{\mc{B}}\right)}\leq \overline{\left(\overline{\mc{A}+\mc{B}}\right)}=\overline{\mc{A}+\mc{B}}.
\end{equation}
From Eqs.(\ref{eq3.61}) and (\ref{eq3.362}), the proof is complete. 
\end{proof}

If $\mathfrak{R}(\mc{B}^T)=\mathfrak{R}(\mc{B}^T\m\mc{A}^T),$ then there exist a tensor $\mc{U}$ such that $\mc{B}=\mc{U}\m\mc{A}\n\mc{B}$ and hence, we obtain  $\mc{B}\m\mc{C}=\mc{U}\m\mc{A}\n\mc{B}\m\mc{C}=\mc{U}\m\mc{A}\n\mc{B}\m\mc{D}=\mc{B}\m\mc{D}.$ This leads the following result.

\begin{theorem}\label{ltcan}
Let $\mc{A}\in\mathbb{R}^{I_1\times \cdots \times  I_M \times J_1\times \cdots \times  J_N}$, $\mc{B}\in\mathbb{R}^{J_1\times \cdots \times J_N\times K_1\times \cdots \times  K_M} $, \\ $\mc{C}\in\mathbb{R}^{K_1\times \cdots \times K_M\times J_1\times \cdots\times  J_N}$ and $\mc{D}\in\mathbb{R}^{K_1\times \cdots \times K_M \times J_1\times \cdots \times  J_N}$ be Boolean tensors with\\ $\mc{A}\n\mc{B}\m\mc{C}=\mc{A}\n\mc{B}\m\mc{D}.$ If $\mathfrak{R}(\mc{B}^T)=\mathfrak{R}(\mc{B}^T\n\mc{A}^T),$ then $\mc{B}\m\mc{C}=\mc{B}\m\mc{D}.$
\end{theorem}

Similar way, we can prove the following corollary.

\begin{corollary}\label{rtcan}
Let $\mc{A}\in\mathbb{R}^{I_1 \times \cdots \times  I_M \times J_1 \times \cdots \times  J_N}$, $\mc{B}\in\mathbb{R}^{J_1\times \cdots \times J_N\times K_1 \times \cdots \times  K_M} $, \\ $\mc{C}\in\mathbb{R}^{K_1\times \cdots \times K_M\times I_1\times \cdots\times  I_M}$ and $\mc{D}\in\mathbb{R}^{K_1 \times \cdots \times  K_M \times I_1\times I_2\times\cdots\times  I_M}$  be  Boolean tensors with $\mc{C}\m\mc{A}\n\mc{B}=\mc{D}\m\mc{A}\n\mc{B}.$ If $\mathfrak{R}(\mc{A})=\mathfrak{R}(\mc{A}\n\mc{B}),$ then $\mc{C}\m\mc{A}=\mc{D}\m\mc{A}.$
\end{corollary}

We now discuss the important result on a transpose of an arbitrary order Boolean tensor, as follows.
\begin{lemma}\label{lemma1}
Let $\mc{A}\in\mathbb{R}^{I_1\times \cdots \times  I_M \times J_1 \times \cdots \times  J_N}$ be any Boolean tensor. Then $\mc{A}\leq\mc{A}\n\mc{A}^T\m\mc{A}.$
 \begin{proof}
 Let $\mc{B} = \mc{A}\n\mc{A}^T\m\mc{A}.$ We need to show that
 \begin{equation*}
  {a}_{i_1\cdots i_M j_1\cdots j_N}\leq {b}_{i_1\cdots i_M j_1\cdots j_N}.
 \end{equation*}
 This inequality is trivial  if  ${a}_{i_1\cdots i_M j_1\cdots j_N}= 0.$ Let us assume ${a}_{i_1\cdots i_M j_1\cdots j_N}=1.$ Now
\begin{equation*}
{b}_{i_1\cdots i_M j_1\cdots j_N} =\sum_{k_1\cdots k_N}\sum_{l_1\cdots l_M}a_{{i_1\cdots i_M}{k_1\cdots k_N}}a_{{l_1\cdots l_M}{k_1\cdots k_N}}a_{{l_1\cdots l_M}{j_1\cdots j_N}}.
\end{equation*}
For $1\leq s\leq N,$ if $k_s=j_s$ and $l_s=i_s,$ then
\begin{equation*}
    {b}_{i_1\cdots i_M j_1\cdots j_N} \geq ({a}_{i_1\cdots i_M j_1\cdots j_N})^3={a}_{i_1\cdots i_M j_1\cdots j_N}=1.
\end{equation*}
Hence the proof is complete.
 \end{proof}
\end{lemma}

\begin{theorem}
Let $\mc{A}\in\mathbb{R}^{I_1\times \cdots \times  I_N\times J_1 \times \cdots \times J_N}$ and $\mc{B}\in\mathbb{R}^{I_1 \times \cdots \times I_N\times J_1 \times \cdots \times J_N}.$ Then the equation $\mc{A}\n\mc{X}=\mc{B}$ is solvable if and only if $\mc{X}=\mc{C},$ where 
$$ c_{i_1\cdots i_N j_1\cdots j_n}=\left\{\begin{array}{cc}
   1  &  \mbox{ if } a_{i_1\cdots i
  _Ni_1\cdots i
  _N}=0 \mbox{ or } b_{i_1\cdots i
  _Nj_1\cdots j
  _N}=1 \mbox{ for all } i_k,~1\leq k\leq N,\\
   0  & otherwise.
\end{array}\right.
$$
\end{theorem}

\begin{proof}
Let $\mc{A}\n\mc{X}=\mc{B}$ is solvable and $\mc{A}\n\mc{X}=D.$ To claim $\mc{D}=\mc{B},$ it is enough to show $d_{i_1\cdots i
  _Nj_1\cdots j
  _N}=1$ if and only if $b_{i_1\cdots i
  _Nj_1\cdots j
  _N}=1.$ Let $d_{i_1\cdots i
  _Nj_1\cdots j
  _N}=1.$ This implies $a_{i_1\cdots i
  _Np_1\cdots p
  _N}=1$ and $c_{p_1\cdots p
  _Nj_1\cdots j
  _N}=1$  for some $p_k~~1\leq k\leq N.$  The condition $c_{p_1\cdots p
  _Nj_1\cdots j
  _N}=1$ yields either $a_{i_1\cdots i
  _Np_1\cdots p
  _N}=0 $  or $b_{p_1\cdots p
  _Nj_1\cdots j
  _N}=1$ for all $p_k~~1\leq k\leq N.$  Since $a_{i_1\cdots i
  _Np_1\cdots p
  _N}=1$ which makes $b_{p_1\cdots p
  _Nj_1\cdots j
  _N}=1$ for all $p_k,~~1\leq k\leq N.$ Therefore $b_{i_1\cdots i
  _Nj_1\cdots j
  _N}=1.$ Now if $b_{i_1\cdots i
  _Nj_1\cdots j
  _N}=1,$ then $a_{i_1\cdots i
  _Nr_1\cdots r
  _N}=1$ and $x_{r_1\cdots r
  _Nj_1\cdots j
  _N}=1$  for some $r_k,~~1\leq k\leq N.$  Suppose $c_{r_1\cdots r
  _Nj_1\cdots j
  _N}=0.$ Then $a_{q_1\cdots q
  _Nr_1\cdots r
  _N}=1$ and $b_{q_1\cdots q
  _Nj_1\cdots j
  _N}=0$  for some $q_k,~~1\leq k\leq N.$ Combining $a_{q_1\cdots q
  _Nr_1\cdots r
  _N}=1$ and $x_{r_1\cdots r
  _Nj_1\cdots j
  _N}=1$, we get $b_{q_1\cdots q
  _Nj_1\cdots j
  _N}=1.$ Which is the contradiction. So $c_{r_1\cdots r
  _Nj_1\cdots j
  _N}=1$ and hence $d_{i_1\cdots i
  _Nj_1\cdots j
  _N}=1.$ The converse part is trivial. 
\end{proof}

In view of the Definition \ref{CompDef} the following theorem is true for Boolean tensors.

\begin{preposition}\label{equitc}
Let $\mc{A} \in\mathbb{R}^{I_1\times \cdots \times I_M\times J_1\times \cdots \times J_N}$ be a Boolean tensor,  then
\begin{enumerate}
    \item[(a)] $(\mc{A}^C)^C=\mc{A};$
    \item[(b)] $(\mc{A}^C)^T=(\mc{A}^T)^C = \mc{A}^{CT}.$
    \end{enumerate}
\end{preposition}

 \begin{remark}
In general $ \mc{B}^C *_N \mc{A}^C \neq (\mc{A}*_N\mc{B})^C \neq \mc{A}^C *_N \mc{B}^C$ for any two tensor $\mc{A},~\mc{B} \in\mathbb{R}^{I_1\times\cdots \times I_M\times I_1\times\cdots \times I_M}$
 \end{remark}

\begin{example}
Consider two Boolean tensor
$~\mc{A}=(a_{ijkl}) \in \mathbb{R}^{{2\times3}\times{2 \rtimes 3}}$ and $~\mc{B}=(b_{ijkl}) \in \mathbb{R}^{{2\times3}\times{2 \rtimes 3}}$ such that
\begin{eqnarray*}
a_{ij11} =
    \begin{pmatrix}
    1 & 0 & 0 \\
    0 & 0 &  1
    \end{pmatrix},
a_{ij12} =a_{ij13} =a_{ij21}=a_{ij22}=a_{ij23}=
    \begin{pmatrix}
     0 & 0 & 0\\
     0 & 0 & 1
    \end{pmatrix}, \mbox{ and }
\end{eqnarray*}
\begin{eqnarray*}
b_{ij11} =b_{ij12}=b_{ij13}=b_{ij21}=b_{ij22}=b_{ij23}=
    \begin{pmatrix}
    1 & 0 & 0 \\
    0 & 0 &  0
    \end{pmatrix}.
\end{eqnarray*}
It is easy to verify $ \mc{B}^C \2 \mc{A}^C \neq (\mc{A}\2\mc{B})^C \neq \mc{A}^C \2 \mc{B}^C$, where
$~(\mc{A}*_2\mc{B})^C=\mc{X}=(x_{ijkl}) \in \mathbb{R}^{{2\times3}\times{2 \rtimes 3}}$ $~\mc{A}^C*_2\mc{B}^C=\mc{Y} =(y_{ijkl}) \in \mathbb{R}^{{2\times3}\times{2 \rtimes 3}}$ and 
$~\mc{B}^C*_2\mc{A}^C=\mc{Z} =(z_{ijkl}) \in \mathbb{R}^{{2\times3}\times{2 \rtimes 3}},$ where
\begin{eqnarray*}
x_{ij11} =x_{ij12} =x_{ij13} =x_{ij21} =x_{ij22} =x_{ij23} =
    \begin{pmatrix}
    0 & 1 & 1 \\
    1 & 1 &  0
    \end{pmatrix},
\end{eqnarray*}
\begin{eqnarray*}
y_{ij11} =y_{ij12} =y_{ij13} =y_{ij21} =y_{ij22} =y_{ij23} =
    \begin{pmatrix}
    1 & 1 & 1 \\
    1 & 1 & 0
    \end{pmatrix}, \mbox{ and }
\end{eqnarray*}
\begin{eqnarray*}
z_{ij11} =z_{ij12} =z_{ij13} =z_{ij21} = z_{ij22} =z_{ij23} =
    \begin{pmatrix}
    0 & 1 & 1 \\
    1 & 1 &  1
    \end{pmatrix}.
\end{eqnarray*}
\end{example}

The next result is the one of the important tool to prove trace of a Boolean tensor.

\begin{theorem}
 Let $\mc{A}\in\mathbb{R}^{I_1 \times \cdots \times  I_N\times J_1 \times \cdots \times J_N}.$ Then $\mc{A}\n\mc{A}^C=\mc{O}$ if and only either $\mc{A}=\mc{O}$ or $\mc{A}=\mc{O}^C.$
\end{theorem}
 \begin{proof}
 Since the converse part is trivial, it is enough to show the sufficient part only. Let $\mc{A}\n\mc{A}^C=\mc{O}.$ Thus 
$\displaystyle\sum_{k_1\cdots k_N}a_{i_1\cdots i_Nk_1\cdots k_N}a_{k_1\cdots k_Nj_1\cdots j_N}^c=0$. This implies,  $ a_{i_1\cdots i_Nk_1\cdots k_N}a_{k_1\cdots k_Nj_1\cdots j_N}^c=0 $ for all $k_s,~1\leq s\leq N.$ Which again yields either $a_{i_1\cdots i_Nk_1\cdots k_N}=0 $  for all $i_s,~k_s$  or $a_{k_1\cdots k_Nj_1\cdots j_N}^c=0$  for all $j_s,~k_s,~1\leq s\leq N.$. Therefore either $\mc{A}=\mc{O}$ or $\mc{A}^C=\mc{O}.$ Hence completes the proof.
 \end{proof}
 
Further, when $\mc{A}\in\mathbb{R}^{I_1 \times \cdots \times I_N\times J_1 \times \cdots \times J_N}$ is  symmetric Boolean tensor, one can write 
 \begin{eqnarray*}
 tr(\mc{A}\n\mc{A}^C)&=&\sum_{i_1\cdots i_N}\sum_{k_1\cdots k_N}a_{i_1\cdots i_Nk_1\cdots k_N}a_{k_1\cdots k_Ni_1\cdots i_N}^c\\
 &=&\sum_{i_1\cdots i_N}\sum_{k_1\cdot, k_N}a_{k_1\cdots k_Ni_1\cdots i_N}^ca_{i_1\cdots i_Nk_1\cdots k_N}\\
  &=&\sum_{k_1\cdots k_N}\sum_{i_1\cdots  i_N}a_{k_1\cdots k_Ni_1\cdots i_N}^ca_{i_1\cdots i_Nk_1\cdots k_N}\\
&=&tr(\mc{A}^C\n\mc{A}).
  \end{eqnarray*}

 Hence, the tensors in the trace of a product of symmetric tensor and its complement can be switched without changing the result. This is stated in the next result.
 
 \begin{theorem}\label{trresult}
  Let $\mc{A}\in\mathbb{R}^{I_1 \times \cdots \times I_N\times J_1 \times \cdots \times J_N}.$ If $\mc{A}$ is symmetric, then 
 $$
 tr(\mc{A}\n\mc{A}^C)=tr(\mc{A}^C\n\mc{A}).
 $$
 \end{theorem}

 \begin{remark}\label{rmk3.10}
  In addition to the result of Theorem \ref{trresult} one can write  $tr(\mc{A}\n\mc{A}^C)=tr(\mc{A}^C\n\mc{A}) = 0$ . 
Further, the symmetricity condition in Theorem \ref{trresult} is only sufficient  but not necessary.
 \end{remark}
  One can verify the Remark \ref{rmk3.10} by the following example.
 
 \begin{example}
 Let a Boolean tensor
$~\mc{A}=(a_{ijkl}) \in \mathbb{R}^{{2\times2}\times{2 \rtimes 2}}$ such that
\begin{eqnarray*}
a_{ij11} =
    \begin{pmatrix}
    1 & 0 \\
    1 & 0
    \end{pmatrix},
a_{ij12} =
    \begin{pmatrix}
     1 & 0 \\
     1 & 0 
    \end{pmatrix},
a_{ij21} =
    \begin{pmatrix}
     1 & 0 \\
     1 & 0
    \end{pmatrix},
a_{ij22} =
    \begin{pmatrix}
    1 & 0 \\
    1 & 0
    \end{pmatrix}.
\end{eqnarray*}
It is clear that $\mc{A}$ is not symmetric but $tr(\mc{A}*_2\mc{A}^C) = tr(\mc{A}^C*_2\mc{A}) = 2$, where $~\mc{A}*_2\mc{A}^C=(x_{ijkl}) \in \mathbb{R}^{{2\times2}\times{2 \rtimes 2}}$ and $~\mc{A}^C*_2\mc{A}=(y_{ijkl}) \in \mathbb{R}^{{2\times2}\times{2 \rtimes 2}}$ with entries
\begin{eqnarray*}
x_{ij11} =
    \begin{pmatrix}
    1 & 0 \\
    1 & 0
    \end{pmatrix},
x_{ij12} =
    \begin{pmatrix}
     1 & 0 \\
     1 & 0 
    \end{pmatrix},
x_{ij21} =
    \begin{pmatrix}
     1 & 0 \\
     1 & 0
    \end{pmatrix},
x_{ij22} =
    \begin{pmatrix}
    1 & 0 \\
    1 & 0
    \end{pmatrix},
\end{eqnarray*}
\begin{eqnarray*}
y_{ij11} =
    \begin{pmatrix}
    0 & 1 \\
    0 & 1
    \end{pmatrix},
y_{ij12} =
    \begin{pmatrix}
     0 & 1 \\
     0 & 1 
    \end{pmatrix},
y_{ij21} =
    \begin{pmatrix}
     0 & 1 \\
     0 & 1
    \end{pmatrix},
y_{ij22} =
    \begin{pmatrix}
    0 & 1 \\
    0 & 1
    \end{pmatrix}.
\end{eqnarray*}
 \end{example}

 Using the complement of a tensor, we now prove the following result.
 
 \begin{lemma}\label{complem}
 Let $\mc{A}\in\mathbb{R}^{I_1 \times \cdots \times I_M\times J_1 \times \cdots \times  J_N},$ $\mc{B}\in\mathbb{R}^{J_1 \times \cdots \times  J_N\times K_1 \times \cdots \times  K_L}$ and $\mc{C}\in\mathbb{R}^{K_1 \times \cdots \times K_L\times J_1 \times \cdots \times J_N}$ be Boolean tensors. Then 
 $$\mc{A}\n\mc{B}\kl\mc{C}\leq \mc{I}^C~~if ~and ~only ~if ~~\mc{A}^C\geq (\mc{B}\kl\mc{C})^T.$$
 \end{lemma}
 \begin{proof}
 $\mc{A}\n\mc{B}\kl\mc{C}\leq \mc{I}^C$  if and only $ \sum_{j_1\cdots j_N}\sum_{k_1\cdots k_L}a_{{i_1\cdots i_M}{j_1\cdots j_N}}b_{{j_1\cdots j_N}{k_1\cdots k_L}}c_{{k_1\cdots k_L}{i_1\cdots i_M}}=0$
for all $i_r,$ $1\leq r\leq M.$ This is equivalent to $a_{{i_1\cdots i_M}{j_1\cdots j_N}}b_{{j_1\cdots j_N}{k_1\cdots k_L}}c_{{k_1\cdots k_L}{i_1\cdots i_M}}=0$ 
for all $i_r,$ $j_s$ and $k_t,$ where $1\leq r\leq M,$ $1\leq s\leq N,$ $1\leq t\leq L.$ This in turn is true if and only
\begin{eqnarray*}
\left(a_{{i_1\cdots i_M}{j_1\cdots j_N}}^c\right)&\geq& b_{{j_1\cdots j_N}{k_1\cdots k_L}}c_{{k_1\cdots k_L}{i_1\cdots i_M}}  \mbox{ for all  $k_t$ }\\
 &=&\left(c_{{i_1\cdots i_M}{k_1\cdots k_L}}^t\right) \left(b_{{k_1\cdots k_L}{j_1\cdots j_N}}^t\right)~\mbox{ for all $k_t$}\\
  &\geq& \sum_{k_1\cdots k_L}\left(c_{{i_1\cdots i_M}{k_1\cdots k_L}}^t\right) \left(b_{{k_1\cdots k_L}{j_1\cdots j_N}}^t\right)\\
    &\geq& \left(\mc{C}^T\n\mc{B}^T\right)_{{i_1\cdots i_M}{j_1\cdots j_N}}=\left(\left(\mc{B}\kl\mc{C}\right)^T\right)_{{i_1\cdots i_M}{j_1\cdots j_N}}.
 \end{eqnarray*}
 Thus the proof is complete.
 \end{proof}
 Now we discuss the important result based on transpose and component-wise complement of an arbitrary order  Boolean tensor, as follows.
\begin{theorem}\label{lucethm}
Let $\mc{A}\in\mathbb{R}^{I_1 \times \cdots \times  I_M\times J_1 \times \cdots \times  J_N}.$ Then $\mc{X}\m\mc{A}\leq \mc{B}$ if and only if $\mc{X}\leq  \left(\mc{B}^C\n\mc{A}^T\right)^C$, and $\mc{A}\n\mc{X}\leq \mc{B}$ if and only if $\mc{X}\leq \left(\mc{A}^T \m \mc{B}^C\right)^C$.
\end{theorem}
\begin{proof}
Let $\mc{X}\m\mc{A}\leq \mc{B}.$ This yields 
$\sum_{k_1\cdots k_M}x_{i_1\cdots i_M k_1\cdots k_M}a_{k_1\cdots k_M j_1\cdots j_N }\leq b_{i_1\cdots i_M j_1\cdots j_N }$ for all $i_r,~(1\leq r\leq M)$ and $j_s,~(1\leq s\leq N).$ This is equivalent to $x_{i_1\cdots i_M k_1\cdots k_M}a_{k_1\cdots k_M j_1\cdots j_N }\leq b_{i_1\cdots i_M j_1\cdots j_N }$  for all $i_r$ and $j_s$ and $k_t~(1\leq t\leq M).$ This in turns is true {\it if and only if} $x_{i_1\cdots i_M k_1\cdots k_M}a_{k_1\cdots k_M j_1\cdots j_N } b_{i_1\cdots i_M j_1\cdots j_N }^c=0,$ for all $j_s$ and $k_t.$ Which is equivalent to\\ $x_{i_1\cdots i_M k_1\cdots k_M}a_{k_1\cdots k_M j_1\cdots j_N } \{b_{j_1\cdots j_N i_1\cdots i_M }^t\}^c=0$ for all $j_s$ and $k_t.$ Summing over all $j_s$ and $k_t,$ we get, $\sum_{k_1\cdots k_M}\sum_{j_1\cdots j_N} x_{i_1\cdots i_M k_1\cdots k_M}a_{k_1\cdots k_M j_1\cdots j_N } \{b_{j_1\cdots j_N i_1\cdots i_M }^t\}^c=0.$ This is true if and only $\mc{X}\m\mc{A}\n\left(\mc{B}^T\right)^C\leq \mc{I}^C.$ By Preposition \ref{equitc} $(a)$, this is equivalent to $\mc{X}\m\mc{A}\n\left(\mc{B}^C\right)^T\leq \mc{I}^C.$ By Lemma \ref{complem}, this in turns true if and only $\mc{X}^C\geq \left(\mc{A}\n(\mc{B}^C)^T\right)^T,$ that is, {\it if and only if} $ \mc{X}\leq \left(\mc{B}^C\m\mc{A}^T\right)^C.$

This completes first part of the theorem. Similar way, we can show the second part of the theorem.
\end{proof}
\begin{corollary}
Let $\mc{E}=\mc{O}^C,$ where $\mc{O}$ is the zero tensor. Then the  following statements are equivalent:
\begin{enumerate}
    \item[(a)] $\mc{X}\m\mc{A}=\mc{O};$
     \item[(b)] $\mc{X}\leq \left(\left(\mc{A}\n\mc{E}\right)^T\right)^C;$
      \item[(c)] $\mc{E}\n\mc{X}\leq \left(\left(\mc{A}\n\mc{E}\right)^T\right)^C.$
\end{enumerate}
\end{corollary}
The same result is also true for $\mc{A}\n\mc{X}=\mc{O}.$  Also the following corollary easily follow from Theorem \ref{lucethm}.

\begin{corollary}
Let $\mc{A}\in\mathbb{R}^{I_1 \times \cdots \times I_M\times J_1 \times \cdots \times  J_N}$ and  $\mc{X}\in\mathbb{R}^{I_1 \times \cdots \times I_M\times I_1 \times \cdots \times  I_M}.$  Then $\mc{X}\m\mc{A}=\mc{B}$ has a solution if and only if $\mc{B}\leq \left(\mc{B}^C\n\mc{A}^T\right)^C\m\mc{A}.$
\end{corollary}

\subsection{Generalized inverses of Boolean tensors}
For the generalization of the generalized inverses of Boolean matrix \cite{rao}, we introduce the definition of $\{i\}$-inverses $(i = 1, 2, 3, 4)$ and the Moore-Penrose inverse of Boolean tensors via the Einstein product, as follows.

 \begin{definition}\label{defgi}
 For any Boolean tensor $\mc{A} \in \mathbb{R}^{I_1\times\cdots\times I_M \times J_1 \times\cdots\times J_N},$  consider the following equations in $\mc{X} \in
\mathbb{R}^{J_1\times\cdots\times J_N \times I_1 \times\cdots\times
I_M}:$
\vspace{-.4cm}
\begin{eqnarray*}
&&(1)~\mc{A}\n\mc{X}\m\mc{A} = \mc{A},\\
&&(2)~\mc{X}\m\mc{A}\n\mc{X} = \mc{X},\\
&&(3)~(\mc{A}\n\mc{X})^T = \mc{A}\n\mc{X},\\
&&(4)~(\mc{X}\m\mc{A})^T = \mc{X}\m\mc{A}.
\end{eqnarray*}
\vspace{-.34cm}
Then $\mc{X}$ is called
\begin{enumerate}
\item[(a)] 
a generalized inverse of $\mc{A}$ if it satisfies $(1)$ and denoted by $\mc{A}^{(1)}.$
\item[(b)]  a reflexive generalized inverse of $\mc{A}$ if it satisfies $(1)$ and $(2)$, which is denoted by $\mc{A}^{(1,2)}.$
\item[(c)] a $\{1,3\}$ inverse of $\mc{A}$ if it satisfies $(1)$ and $(3)$, which is denoted by $\mc{A}^{(1,3)}.$
\item[(d)]  a $\{1,4\}$ inverse of $\mc{A}$ if it satisfies $(1)$ and $(4)$, which is  denoted by $\mc{A}^{(1,4)}.$
\item[(e)] the Moore-Penrose inverse of $\mc{A}$ if it satisfies all four conditions $[(1)-(4)]$, which is denoted by $\mc{A}^{\dagger}.$
\end{enumerate}
\end{definition}

The following remark and corollary are follows from the Definition \ref{defgi}.

\begin{remark}\label{rm11}
 If $\mc{X}$ is the generalized inverse of a Boolean tensor $\mc{A} \in\mathbb{R}^{I_1 \times \cdots \times  I_M\times J_1 \times \cdots \times J_N}$ then $\mc{X}\m\mc{A}\n\mc{X}$ is the reflexive generalized inverse of $\mc{A}.$
\end{remark}

\begin{corollary}\label{corm1}
If $\mc{X}$ is the generalized inverse of a Boolean tensor $\mc{A} \in\mathbb{R}^{I_1 \times \cdots \times  I_M\times J_1 \times \cdots \times J_N}$ Then   
\begin{enumerate}
    \item[(a)] $\mc{X}^T$ is the generalized inverse of $\mc{A}^T;$
    \item[(b)] $(\mc{X}_1+\mc{X}_2)$ is the generalized inverse of of a Boolean tensor $\mc{A}$ when $\mc{X}_1$ and $\mc{X}_2$ are two generalized inverse of $\mc{A}.$
\end{enumerate}
\end{corollary}

Thus the existence of generalized inverse of a Boolean tensor guarantees the existence of a reflexive generalized inverse. In addition to that, the Remark \ref{rm11} and Corollary \ref{corm1} (b) ensures that the existence of one-generalized inverse implies the existence of finite number generalized inverses. In view of the fact, we define the maximum generalized inverse of a Boolean tensor, as follow:

\begin{definition}
 Let $\mc{A}\in\mathbb{R}^{I_1\times \cdots \times I_M\times J_1 \times \cdots \times J_N}.$ A tensor $\mc{X}$ is called maximum generalized inverse of $\mc{A}$ if $\mc{G}\leq \mc{X}$ for every generalized inverse $\mc{G}$ of $\mc{A}.$
\end{definition}

Note that, the generalized inverse of a Boolean tensor need not be unique which explained in the next example.

\begin{example}\label{example18}
Consider a Boolean tensor
$~\mc{A}=(a_{ijkl}) \in \mathbb{R}^{{2\times3}\times{2 \rtimes 3}}$ with entries
\begin{eqnarray*}
a_{ij11} =a_{ij12} =a_{ij13} =a_{ij21} =a_{ij22} =a_{ij23} =
    \begin{pmatrix}
    1 & 0 & 0 \\
    1 & 0 &  0
    \end{pmatrix}.
\end{eqnarray*}
Then it can be easily verified that both tensors
$~\mc{X}=(x_{ijkl}) \in \mathbb{R}^{{2\times3}\times{2 \rtimes 3}}$  and $~\mc{Y}=(y_{ijkl}) \in \mathbb{R}^{{2\times3}\times{2 \rtimes 3}}$  with entries
\begin{eqnarray*}
x_{ij11} =
    \begin{pmatrix}
    0 & 1 & 1 \\
    1 & 1 &  1
    \end{pmatrix},
x_{ij12} =x_{ij13} =x_{ij21} = x_{ij22} =x_{ij23} =
    \begin{pmatrix}
     0 & 0 & 0\\
     0 & 0 & 0
    \end{pmatrix},\mbox{ and }
   \end{eqnarray*}
\begin{eqnarray*}
y_{ij11} =
    \begin{pmatrix}
    1 & 0 & 0 \\
    0 & 0 &  1
    \end{pmatrix},
y_{ij12} =y_{ij13} =y_{ij21} =y_{ij22} =y_{ij23} =
    \begin{pmatrix}
     0 & 0 & 0\\
     0 & 0 & 0
    \end{pmatrix},
\end{eqnarray*}
are satisfies the required condition of the Definition \ref{defgi}.
\end{example}
 
 Foa a Boolean tensor $\mc{A}\in\mathbb{R}^{I_1 \times \cdots \times  I_N \times J_1 \times \cdots \times   J_N},$ the number of generalized inverses are finite and the maximum number of generalized inverses is $2^{I_1 \times \cdots \times  I_N\times I_1 \times \cdots \times I_N}.$
 The next result assures the uniqueness and is true only for invertiable  tensors.
 \begin{lemma}
 Let $\mc{A}\in\mathbb{R}^{I_1 \times \cdots \times  I_N\times I_1 \times \cdots \times  I_N}$ be any Boolean tensor. If $\mc{A}$ is invertiable then $\mc{A}^{-1}$ is the only generalized inverse of $\mc{A}.$
 \end{lemma}
 Next, we discus the equivalence condition for consistent system and generalized inverse. 
 
 \begin{theorem}\label{eqgen}
Let $\mc{A}\in\mathbb{R}^{I_1 \times \cdots \times  I_M\times J_1 \times \cdots \times J_N}$ and $\mc{X}\in\mathbb{R}^{J_1 \times \cdots \times J_N\times I_1 \times \cdots \times  I_M}.$  Then the followings are equivalent:
\begin{enumerate}
    \item[(a)] $\mc{A}\n\mc{X}\m\mc{A}=\mc{A}.$
    \item[(b)] $\mc{X}\m\mc{Y}$ is a solution of the tensor equation $\mc{A}\n\mc{Z}=\mc{Y}$ whenever $\mc{Y}\in\mathfrak{R}(\mc{A}).$
    \item[(c)] $\mc{A}\n\mc{X}$ is idempotent and $\mathfrak{R}(\mc{A})=\mathfrak{R}(\mc{A}\n\mc{X}).$
     \item[(d)] $\mc{X}\m\mc{A}$ is idempotent and $\mathfrak{R}(\mc{A^T})=\mathfrak{R}(\mc{A}^T\n\mc{X}^T).$
\end{enumerate}
\end{theorem}

\begin{proof}
First we will claim $(a)$ {\it if and only if} $(b).$ Let us assume $(a)$ holds and $\mc{Y}\in\mathfrak{R}(\mc{A}).$ Then there exists a Boolean tensor $\mc{Z}\in \mathbb{R}^{J_1\times J_2\times\cdots\times  J_N}$ such that $\mc{A}\n\mc{Z}=\mc{Y}.$ Now 
$$\mc{A}\n\mc{X}\m\mc{Y}=\mc{A}\n\mc{X}\m\mc{A}\n\mc{Z}=\mc{A}\n\mc{Z}=\mc{Y}.$$ Therefore, $\mc{X}\m\mc{Y}$ is a solution of $\mc{A}\n\mc{Z}=\mc{Y}.$ Conversely assume $(b)$ is true. That is $\mc{A}\n\mc{X}\m\mc{Y}=\mc{Y}$ for all $\mc{Y}\in\mathfrak{R}(\mc{A}).$ Since $\mc{Y}\in\mathfrak{R}(\mc{A})$ which implies there exists $\mc{U}\in \mathbb{R}^{J_1\times\cdots \times  J_N}$ such that $\mc{A}\n\mc{U}=\mc{Y}.$ Thus $\mc{A}\n\mc{X}\m\mc{A}\n\mc{U}=\mc{A}\n\mc{U}$ for all $\mc{U}\in \mathbb{R}^{J_1\times\cdots\times  J_N}.$ Therefore $\mc{A}\n\mc{X}\m\mc{A}=\mc{A}.$ Next we show the equivalence between $(a)$ and $(c).$ Clearly $(a)$ implies $\mc{A}\n\mc{X}$ idempotent. Since $\mc{A}=\mc{A}\n\mc{X}\m\mc{A}$ and $\mc{A}\n\mc{X}=\mc{A}\n {\mc{X}\m\mc{A}\n\mc{X}},$ so by Lemma \ref{range-stan} $\mathfrak{R}(\mc{A})=\mathfrak{R}(\mc{A}\n\mc{X}).$ Using the same idea, we can easily show the equivalence between $(a)$ and $(d).$ Hence completes the proof. 
\end{proof}

Since $\mc{A}^T\m\mc{A}\n\mc{X}_1\n\mc{A}^T\m\mc{A}=\mc{A}^T\m\mc{A}$, so by Theorem \ref{ltcan}, $\mc{A}\n\mc{X}_1\n\mc{A}^T\m\mc{A}=\mc{A}.$ Which leads the following corollary. 

\begin{corollary}
Let $\mathfrak{R}(\mc{A}^T)=\mathfrak{R}(\mc{A}^T\m\mc{A}).$ If $\mc{X}_1$ and $\mc{X}_2$ are generalized inverses of $\mc{A}^T\m\mc{A}$ and $\mc{A}\n\mc{A}^T$ respectively, then $\mc{X}_1\n\mc{A}^T$ and $\mc{A}^T\m\mc{X}_2$ are generalized inverse of $\mc{A}.$
\end{corollary}

Further, from the range conditions, if $\mathfrak{R}(\mc{A}^T)\subseteq \mathfrak{R}(\mc{B}^T)$ and $\mathfrak{R}(\mc{C})\subseteq \mathfrak{R}(\mc{B}).$ Then  $\mc{A}=\mc{V}\m\mc{B}$ and $\mc{C}=\mc{B}\n\mc{U}$ for some tensors $\mc{U}$ and $\mc{V}.$ Now 
$\mc{A}\n\mc{X}\m\mc{C}= \mc{V}\m\mc{B}\n\mc{X}\m\mc{B}\n\mc{U}=\mc{V}\m\mc{B}\n\mc{U}$ which does not rely on $\mc{X}$. So it is invariant to the choice of $\mc{X}.$ So, we conclude this observation in the following corollary.

\begin{corollary}
Let $\mc{A},$ $\mc{B}$ and $\mc{C}$ be suitable tensors such that $\mathfrak{R}(\mc{A}^T)\subseteq \mathfrak{R}(\mc{B}^T)$ and $\mathfrak{R}(\mc{C})\subseteq \mathfrak{R}(\mc{B}).$ If the generalized inverse of $\mc{B}$ exists, then $\mc{A}\n\mc{X}\m\mc{C}$ is invariant to $\mc{X},$ where $\mc{X}$ is the generalized inverse  of $\mc{B}.$
\end{corollary}

To prove the next result, we define regular and singular of tensors, i.e., 
A tensor $\mc{A}\in\mathbb{R}^{I_1\times \cdots \times I_M\times J_1\times \cdots J_N},$ is called regular if the tensor equation $\mc{A}\n\mc{X}\m\mc{A}=\mc{A}$ has a solution, otherwise called \textit{singular}.

\begin{theorem}\label{thm3.43}
Let $\mc{A}\in\mathbb{R}^{I_1 \times \cdots \times  I_M\times J_1 \times \cdots \times J_N},$  $\mc{S}\in\mathbb{R}^{I_1 \times \cdots \times  I_M\times I_1 \times \cdots \times I_M},$ and $\mc{T}\in\mathbb{R}^{J_1 \times \cdots \times  J_M\times J_1 \times \cdots \times  J_N}.$ If $\mc{S}$ and $\mc{T}$ are invertible, then the following are equivalent:
\begin{enumerate}
    \item[(a)] $\mc{A}$ is regular.
     \item[(b)] $\mc{S}\m\mc{A}\n\mc{T}$ is regular.
     \item[(c)] $\mc{A}^T$ is regular.
       \item[(d)] $\mc{T}\n\mc{A}^T\m\mc{S}$ is regular.
\end{enumerate}
\end{theorem}

Based on the block tensor\cite{sun} and their properties, we have the following lemma.
\begin{lemma}\label{block}
Let $\mc{A}\in\mathbb{R}^{I_1 \times \cdots \times I_M\times J_1 \times \cdots \times  J_N}.$ Then $\mc{A}$ is regular if and only if $\begin{bmatrix}
\mc{A} &\mc{O}\\
\mc{O} & \mc{B}
\end{bmatrix}$ is regular for all regular tensors $\mc{B}\in\mathbb{R}^{I_1 \times \cdots \times I_M\times J_1 \times \cdots \times  J_N}.$
\end{lemma}

\begin{proof}
Let $\mc{A}$ and $\mc{B}$ be regular tensors. Then there exist tensors $\mc{X}$ and $\mc{Y}$ such that $\mc{A}\n\mc{X}\m\mc{A}=\mc{A}$ and $\mc{B}\n\mc{Y}\m\mc{B}=\mc{B}.$ Let $\mc{Z}=\begin{bmatrix}
\mc{X} & \mc{O}\\
\mc{O} & \mc{Y}\\
\end{bmatrix}.$ Now 
\begin{eqnarray*}
\begin{bmatrix}
\mc{A} & \mc{O}\\
\mc{O} & \mc{B}\\
\end{bmatrix}\n\mc{Z}\m\begin{bmatrix}
\mc{A} & \mc{O}\\
\mc{O} & \mc{B}\\
\end{bmatrix}&=&\begin{bmatrix}
\mc{A} & \mc{O}\\
\mc{O} & \mc{B}\\
\end{bmatrix}\n\begin{bmatrix}
\mc{X} & \mc{O}\\
\mc{O} & \mc{Y}\\
\end{bmatrix}\m\begin{bmatrix}
\mc{A} & \mc{O}\\
\mc{O} & \mc{B}\\
\end{bmatrix}\\
&=&
\begin{bmatrix}
\mc{A}\n\mc{X} & \mc{O}\\
\mc{O} & \mc{B}\n\mc{Y}\\
\end{bmatrix}\m\begin{bmatrix}
\mc{A} & \mc{O}\\
\mc{O} & \mc{B}\\
\end{bmatrix}\\
&=&\begin{bmatrix}
\mc{A}\n\mc{X}\m\mc{A} & \mc{O}\\
\mc{O} & \mc{B}\n\mc{Y}\m\mc{B}\\
\end{bmatrix}=\begin{bmatrix}
\mc{A} & \mc{O}\\
\mc{O} & \mc{B}\\
\end{bmatrix}.
\end{eqnarray*}
Thus $\begin{bmatrix}
\mc{A} & \mc{O}\\
\mc{O} & \mc{B}\\
\end{bmatrix}$ is regular. The converse part can be proved in the similar way. 
\end{proof}

   We now present another characterization of the generalized inverse of the Boolean tensor, as follows.
\begin{theorem}\label{lemcomp}
Let $\mc{A}\in\mathbb{R}^{I_1 \times \cdots \times  I_M\times J_1 \times \cdots \times  J_N}.$ Then 
$$\mc{A}\n\mc{X}\m\mc{A}\leq \mc{A}~~ if ~and~ only~ if ~~\mc{X}\leq \left(\mc{A}\n\mc{A}^{CT}\m\mc{A}\right)^{CT}.$$
\end{theorem}
\begin{proof}
Applying Theorem \ref{lucethm} repetitively, we get 
$\mc{A}\n\mc{X}\m\mc{A}\leq \mc{A}$ if and only if $\mc{X}\m\mc{A}\leq \left(\mc{A}^T\m\mc{A}^C\right)^C$, which equivalently if and only if 
\begin{equation*}
\mc{X}\leq \left(\left(\left(\mc{A}^T\m\mc{A}^C\right)^C\right)^C\n\mc{A}^T\right)^C = \left(\mc{A}^T\m\mc{A}^C\n\mc{A}^T\right)^C =\left(\mc{A}\n\mc{A}^{CT}\m\mc{A}\right)^{CT}.
\end{equation*}
\end{proof}
Using the Theorem \ref{lemcomp},  and the fact of transpose and component-wise complement of a Boolean tensor, we obtain an important result for finding the maximum generalized inverse of a Boolean tensor. 
\begin{corollary}
Let  $\mc{A}\in\mathbb{R}^{I_1 \times \cdots \times  I_M\times J_1 \times \cdots \times  J_N}$ be  regular. Then the following are holds
\begin{enumerate}
\item[(a)] $\mc{A}=\mc{A}\n\left(\mc{A}\n\mc{A}^{CT}\m\mc{A}\right)^{CT}\m\mc{A};$
    \item[(b)] $\left(\mc{A}\n\mc{A}^{CT}\m\mc{A}\right)^{CT}$ is the maximum generalized inverse of $\mc{A};$ 
    \item[(c)] $\left(\mc{A}\n\mc{A}^{CT}\m\mc{A}\right)^{CT}\m\mc{A}\n\left(\mc{A}\n\mc{A}^{CT}\m\mc{A}\right)^{CT}$ is the maximum reflexive generalized inverse of $\mc{A}.$
\end{enumerate} 
\end{corollary}

Next, we discuss some equivalence results between generalized and other inverses.

\begin{theorem}\label{eqv14}
Let $\mc{A}\in\mathbb{R}^{I_1 \times \cdots \times  I_M\times J_1 \times \cdots \times   J_N}$ be any Boolean tensor, then the following statements are equivalent:
\begin{enumerate}
\item[(a)] $\mc{A}^{(1,4)}$ exists.
\item[(b)] $\mc{A}^{(1)}$ exists and $\mathfrak{R}(\mc{A}) = \mathfrak{R}(\mc{A}\n\mc{A}^T).$ 
\item[(c)] $(\mc{A}\n\mc{A}^T)^{(1)}$ exists and  $\mc{X}\m\mc{A}\n\mc{A}^T = \mc{A}^T$ for some tensor $\mc{X}.$
\end{enumerate}

\begin{proof}
Consider $(a)$ is true and $\mc{A}^{(1,4)}=\mc{X}.$ Existence of  $\mc{A}^{(1)}$ is trivial and hence $\mathfrak{R}(\mc{A}) = \mathfrak{R}(\mc{A}\n\mc{A}^T).$ Now we claim $(b)\Rightarrow (c).$ Let $\mc{A}^{(1)}$ exists and  $\mathfrak{R}(\mc{A}) = \mathfrak{R}(\mc{A}\n\mc{A}^T).$  Then there exist a Boolean tensor  $\mc{U}\in\mathbb{R}^{I_1\times\cdots \times I_M\times J_1\times\cdots\times  J_N}$ such that $\mc{A} = \mc{A}\n\mc{A}^T\m\mc{U}.$  Which implies $\mc{A}\n\mc{A}^{T}=\mc{A}\n\mc{A}^T\m\mc{U}\n\mc{U}^T\m\mc{A}\n\mc{A}^T.$ So generalized inverse of $\mc{A}\n\mc{A}^T$ exists. If we take $\mc{X}=\mc{A}^T\n(\mc{A}\n\mc{A}^T)^{(1)},$ then 
\begin{eqnarray*}
\mc{X}\m\mc{A}\n\mc{A}^T &=&  {\mc{A}^T}\m(\mc{A}\n\mc{A}^T )^{(1)}\m\mc{A}\n \mc{A}^T=\mc{U}^T\m\mc{A}\n\mc{A}^T\m(\mc{A}\n\mc{A}^T )^{(1)}\m\mc{A}\n\mc{A}^T\\
&=&\mc{U}^T\m\mc{A}\n\mc{A}^T=\mc{A}^T.
\end{eqnarray*}
Finally, we claim $(c)\Rightarrow (a).$ Let $\mc{X}\m\mc{A}\n\mc{A}^T = \mc{A}^T$. Taking transpose on both sides, we get $\mc{A}\n\mc{X}\m\mc{A} = \mc{A} $, As
\begin{eqnarray*}
(\mc{X}\m\mc{A})^T &= &\mc{A}^T\m\mc{X}^T = \mc{X}\m\mc{A}\n\mc{A}^T\m\mc{X}^T\\
&=& (\mc{X}\m\mc{A}\n\mc{A}^T\m\mc{X}^T)^T = (\mc{A}^T\m\mc{X}^T)^T = \mc{X}\m\mc{A}. 
\end{eqnarray*}
Thus $\mc{X}=\mc{A}^{(1,4)}.$ Hence the proof is complete.
\end{proof}
\end{theorem}

Using the similar way, we can show the following theorem.

\begin{theorem}\label{eqv13}
Let $\mc{A}$ be any Boolean tensor, then the following statements are equivalent:
\begin{enumerate}
\item[(a)] $\mc{A}^{(1,3)}$ exists.
\item[(b)] $\mc{A}^{(1)}$ exists and $\mathfrak{R}(\mc{A}^T) = \mathfrak{R}(\mc{A}^T\m\mc{A}).$ 
\item[(c)] There exists a Boolean tensor $\mc{X}$ such that $\mc{A}^T=\mc{A}^T\m\mc{A}\n\mc{X}.$ 
\end{enumerate}
\end{theorem}

We now discuss the characterization of Moore-Penrose inverse of Boolean tensors.  The similar proof of Theorem 3.2 in \cite{sun}, we have the uniqueness of the Moore-Penrose inverse of a Boolean tensor in $\mathbb{R}^{I_1\times \cdots \times I_M\times J_1\times\cdots\times  J_N}$, as follows.

\begin{lemma}\label{mpiu}
Let $\mc{A}\in\mathbb{R}^{I_1 \times \cdots \times  I_M\times J_1 \times \cdots \times  J_N}$ be any Boolean tensor. If the Moore-Penrose inverse of $\mc{A}$ exists then it is unique.
\end{lemma}
 In the next lemma, we discuss an estimate of Moore-Penrose inverse a tensor, as follows.
 \begin{lemma}\label{lemma2}
 Let   $\mc{A}\in\mathbb{R}^{I_1 \times \cdots \times  I_M\times J_1 \times \cdots \times  J_N}$ be a Boolean tensor and suppose $\mc{A}$ admits a Moore-Penrose inverse. Then  $\mc{A}\n\mc{A}^T\m\mc{A}\leq\mc{A}$
 \begin{proof}
 Let $\mc{B}= \mc{A}^T\m\mc{A}$. Since $\mc{B}$ is a Boolean tensor of even order and there are finitely many  Boolean tensors of same order, so there must exist positive integers $s,t\in\mathbb{N}$ such that $\mc{B}^s$ = $\mc{B}^{s+t}.$ Without loss of generality, we can assume that $s$ is the smallest positive integer for which $\mc{B}^s= \mc{B}^{s+t}$ for some $t\in \mathbb{N}.$ Now we will show $s =1.$ Suppose $s\geq 2$. Let $\mc{X}$ be the Moore-Penrose inverse of $\mc{A}$. Since $\mc{B}= \mc{A}^T\m\mc{A}$ and $\mc{B}^s= \mc{B}^{s+t}$ which implies $\mc{A}^T{\m}\mc{A}{\n}\mc{B}^{s-1} = \mc{A}^T{\m}\mc{A}{\n}\mc{B}^{s+t-1}.$  Pre-multiplying both side $\mc{X}^T$ yields $\mc{A}{\n}\mc{B}^{s-1} = \mc{A}{\n}\mc{B}^{s+t-1},$  which implies $\mc{A}\n\mc{A}^T{\m}\mc{A}{\n}\mc{B}^{s-2} = \mc{A}\n\mc{A}^T{\m}\mc{A}{\n}\mc{B}^{s+t-2}.$ Further, pre-multiplying both side $\mc{X}$ yields $\mc{A}^T\m\mc{X}^T\n\mc{A}^T{\m}\mc{A}{\n}\mc{B}^{s-2} =\mc{A}^T\m\mc{X}^T\n\mc{A}^T{\m}\mc{A}{\n}\mc{B}^{s+t-2},$  which implies $\mc{B}^{s-1}= \mc{B}^{s+t-1}.$
 
Thus, the minimality of $s$ is false and hence $s=1.$ Therefore $\mc{B}=\mc{B}^{t+1}$ for some $t\in\mathbb{N}.$ Again we have $\mc{B}=\mc{B}^{t+1}$, which implies $\mc{A}^T{\m}\mc{A} = \mc{A}^T{\m}\mc{A}{\n}\mc{B}^{t}$. Pre-multiplying $\mc{X}^T$ both sides yields
$\mc{A}\n\mc{X}{\m}\mc{A} = \mc{A}\n\mc{X}{\m}\mc{A}{\n}\mc{B}^{t}.$
 Thus 
 \begin{equation}\label{eqfl}
 \mc{A} = \mc{A}{\n}\mc{B}^{t}=\mc{A}\n(\mc{A}^T\m\mc{A})^t.
 \end{equation}
Applying Lemma \ref{lemma1} to $\mc{A}\n\mc{A}^T\n\mc{A}$ repetitively and combining Eq. (\ref{eqfl}), we obtain
$$\mc{A}\n\mc{A}^T\m\mc{A}\leq \mc{A}\n(\mc{A}^T\m\mc{A})^2\leq\cdots\leq \mc{A}\n(\mc{A}^T\m\mc{A})^t=\mc{A}.$$
  \end{proof}
\end{lemma}

Using the Lemma \ref{lemma1} and \ref{lemma2} one can obtain  an interesting result on invertibility of Boolean tensor as follows, 

\begin{corollary}
A Boolean tensor $\mc{A}\in\mathbb{R}^{I_1 \times \cdots \times  I_N\times I_1 \times \cdots \times  I_N}$ is invertible if and only if 
 $$\mc{A}\n\mc{A}^T=\mc{A}^T\m\mc{A}=\mc{I}.$$ 
\end{corollary}
From the Definition \ref{perm}, we obtain
\begin{eqnarray*}
(\mc{P}\n\mc{P}^T)_{i_1 \cdots i_Nj_1\cdots j_N}&=&\sum_{k_1\cdots k_N}(\mc{P})_{i_1 \cdots i_Nk_1\cdots k_N}(\mc{P}^T)_{k_1 \cdots k_Nj_1\cdots j_N}\\
&=&\sum_{k_1\cdots k_N}(\mc{P})_{i_1 \cdots i_Nk_1\cdots k_N}(\mc{P})_{j_1 \cdots j_Nk_1\cdots k_N}\\
&=&(\mc{P})_{i_1 \cdots i_N\pi(j_1)\cdots \pi(j_N)}(\mc{P})_{j_1 \cdots j_N\pi(j_1)\cdots \pi(j_N)}\\
&=& \left\{\begin{array}{cc}
   1  & \mbox{ if } i_s=j_s \mbox{ for all } 1\leq s\leq N. \\
    0 & \mbox{ otherwise.}
\end{array}\right.\\
&=&(\mc{I})_{i_1 \cdots i_Nj_1\cdots j_N}.
\end{eqnarray*}
Similar way, we can also show $\mc{P}^T\n\mc{P}=\mc{I}.$ Therefore, every permutation tensors are orthogonal and  invertible. Adopting this result, we now present a characterization of the permutation tensor, as follows.
\begin{preposition}\label{permu}
A Boolean tensor $\mc{A}$ has an inverse {\it if and only if} it is a permutation tensor.
\end{preposition}
Next result contains five equivalent conditions involving the existence of Moore-Penrose inverse of a Boolean tensor.
 \begin{theorem}\label{thm1}
 Let $\mc{A}\in\mathbb{R}^{I_1 \times \cdots \times  I_M \times J_1\times \cdots \times  J_N.}$ be any tensor.  Then  the  following statements are equivalent:
 \begin{enumerate}
 \item[(i)] The  Moore-Penrose  inverse  of $ \mc{A}$ exists and unique.
 \item[(ii)] $\mc{A}{\n}\mc{A}^T{\n}\mc{A} \leq \mc{A}.$
 \item[(iii)] $\mc{A}{\n}\mc{A}^T{\n}\mc{A} = \mc{A}.$
\item[(iv)] The Moore-Penrose inverse of $\mc{A}$ exists and equals $\mc{A}^T$.
\item[(v)] There exist a tensor $\mc{G}$ such that 
  $\mc{G}{\n}\mc{A}{\n}\mc{A}^T=\mc{A}^T$ and $\mc{A}^T{\n}\mc{A}{\n}\mc{G}=\mc{A}^T$.
\end{enumerate}
 \end{theorem}

 \begin{proof}
 If $(i)$ holds then by Lemma \ref{lemma2} $(ii)$ holds. Also $(ii)\Rightarrow (iii)$ by Lemma \ref{lemma1}. The statements $(iii)\Rightarrow (iv)$ and $(iv)\Rightarrow (i)$ are trivial by definition. Now we will show equivalence between $(i)$ and $(v)$. Suppose $(i)$ holds. If we take $\mc{G}=\mc{A}^T$ then $(v)$ hold. Conversely assume $(v)$ is true. To prove Moore-Penrose inverse of $A$ exists, first we show the following results:
 \begin{itemize}
     \item $\mc{A}\n\mc{G}\m\mc{A} = \mc{A}$\\
     Since $\mc{G}\m \mc{A}\n\mc{A}^T=\mc{A}^T$ which implies $ \mc{A}\n\mc{A}^T\m\mc{G}^T = \mc{A}. $ Pre multiplying $\mc{G}$ and post multiplying $\mc{A}^T$ both sides, we obtain $\mc{G}\m\mc{A}\n\mc{A}^T\m\mc{G}^T\n\mc{A}^T=\mc{G}\m\mc{A}\n\mc{A}^T. $ 
     Thus $\mc{A}^T\m\mc{G}^T\n\mc{A}^T=\mc{A}^T.$ 
     Hence $\mc{A}\n\mc{G}\m\mc{A} = \mc{A}.$
      \item $(\mc{G}\m\mc{A})^T = \mc{A}^T\m\mc{G}^T=\mc{G}{\m}\mc{A}{\n}\mc{A}^T\m\mc{G}^T=\mc{G}\m\mc{A}.$ Therefore $\mc{G}\m\mc{A}$ is symmetric.
     \item $(\mc{A}\n\mc{G})^T = \mc{G}^T\n\mc{A}^T=\mc{G}^T\n\mc{A}^T\m\mc{A}\n\mc{G}=\mc{A}\n\mc{G}.$ Thus $\mc{A}\n\mc{G}$ is symmetric.
          \end{itemize}
 Now we will show the tensor $\mc{X}=\mc{G}\m\mc{A}\n\mc{G}$ is the Moore-Penrose of $\mc{A}.$ Since 
 \begin{enumerate}
     \item [$\bullet$] 
  $\mc{A}\n\mc{X}\m\mc{A} =\mc{A}\n\mc{G}\m\mc{A}=\mc{A}.$
   \item [$\bullet$] 
  $\mc{X}\m\mc{A}\n\mc{X}  =\mc{G}\m\mc{A}\n\mc{G}\m\mc{A}\n\mc{G} =\mc{X}.$
   \item [$\bullet$] 
  $(\mc{A}\n\mc{X})^T=(\mc{A}\n\mc{G}\m\mc{A}\n\mc{G})^T=(\mc{A}\n\mc{G})^T\m(\mc{A}\n\mc{G})^T
 =\mc{A}\n\mc{G}\m\mc{A}\n\mc{G}=\mc{A}\n\mc{X}.$
   \item [$\bullet$] 
 $(\mc{X}\m\mc{A})^T=(\mc{G}\m\mc{A}\n\mc{G}\m\mc{A})^T=(\mc{G}\m\mc{A})^T\n(\mc{G}\m\mc{A})^T
 =\mc{G}\m\mc{A}\n\mc{G}\m\mc{A}=\mc{X}\m\mc{A}.$
  \end{enumerate}

 Therefore, $\mc{X}$ is the Moore-Penrose inverse of $\mc{A}$ and By Lemma \ref{mpiu} it is unique.
  \end{proof}

  The reverse order law for the Moore-Penrose inverses of tensors yields a class  of challenging problems that are fundamental research in the theory of generalized inverses. Research on reverse order law tensors has been very active recently \cite{Mispa18, PanRad18} but as per the above theorem it is trivially true in case of Boolean tensors.  

  \begin{remark}
   If Moore-Penrose inverses of $\mc{A}\in\mathbb{R}^{I_1 \times \cdots \times  I_M\times  J_1 \times \cdots \times  J_N}$, $\mc{B}\in\mathbb{R}^{J_1 \times \cdots \times   J_N\times K_1 \times \cdots \times K_L},$ and $\mc{A}\n\mc{B}$ exists, then the reverse-order law for the Moore-Penrose inverse is always exists, i.e., 
   $$(\mc{A}\n\mc{B})^\dagger=\mc{B}^\dagger\n\mc{A}^\dagger.$$
  \end{remark}

\subsection{Weighted Moore-Penrose inverse}

Utilizing the Einstein product, weighted Moore-Penrose inverse of even-order tensor and arbitrary-order tensor was introduced in \cite{BehMM19, we17}, very recently. This work motivate us to study weighted Moore-Penrose inverse for Boolean tensors.

\begin{definition}\label{wmpi}
Let $\mc{A}\in\mathbb{R}^{I_1 \times \cdots \times  I_M\times J_1 \times \cdots \times J_N}$,  $\mc{M}\in\mathbb{R}^{I_1 \times \cdots \times  I_M\times I_1 \times \cdots \times  I_M}$ and $\mc{N}\in\mathbb{R}^{J_1  \times \cdots \times  J_N\times J_1 \times \cdots \times J_N}$ be three Boolean tensors. If a Boolean tensor $\mc{Z}\in\mathbb{R}^{J_1 \times \cdots \times J_N\times I_1 \times \cdots \times   I_M}$ satisfying
\vspace{-.5cm}
\begin{eqnarray*}
&(1)&\mc{A}\n\mc{Z}\m\mc{A} = \mc{A},\\
&(2)&\mc{Z}\m\mc{A}\n\mc{Z} = \mc{Z},\\
&(3)&(\mc{M}\m\mc{A}\n\mc{Z})^T = \mc{M}\m\mc{A}\n\mc{Z},\\
&(4)&(\mc{Z}\m\mc{A}\n\mc{N})^T = \mc{Z}\m\mc{A}\n\mc{N},
\end{eqnarray*}
is called weighted Moore-Penrose inverse of $\mc{A}$ and it is denoted by $A^{\dagger}_{\mc{M},\mc{N}}.$
\end{definition}
Note that, the weighted Moore-Penrose inverse need not be unique in general. This can be verified by the following example.
\begin{example}
Let the Boolean tensor
$~\mc{A}=(a_{ijkl}) \in \mathbb{R}^{{2\times3}\times{2 \rtimes 3}}$ be defined as in Example \ref{example18} with $\mc{N}=\mc{O} \in \mathbb{R}^{{2\times3}\times{2 \rtimes 3}}$ and  $\mc{M}=(m_{ijkl}) \in \mathbb{R}^{{2\times3}\times{2 \rtimes 3}}$ such that 
\begin{eqnarray*}
m_{ij11} =m_{ij21}=
    \begin{pmatrix}
    1 & 0 & 0 \\
    0 & 0 &  0
    \end{pmatrix},
m_{ij12} = m_{ij22} =
    \begin{pmatrix}
     1 & 0 & 0\\
     0 & 0 & 1
    \end{pmatrix},
m_{ij13} =m_{ij23}
    \begin{pmatrix}
     0 & 0 & 0\\
     0 & 0 & 1
\end{pmatrix}.
\end{eqnarray*}
Then it can be easily verified that both $~\mc{X}=(x_{ijkl}) \in \mathbb{R}^{{2\times3}\times{2 \rtimes 3}}$, $~\mc{Y}=(y_{ijkl}) \in \mathbb{R}^{{2\times3}\times{2 \rtimes 3}}$ defined in Example \ref{example18} satisfies all conditions of Definition \ref{wmpi}. \end{example}
The uniqueness and existence of weighted Moore-Penrose inverse and some equivalent properties will be discussed in the next part of this subsection.

 \begin{theorem}\label{uwmpi}
Let $\mc{A}\in\mathbb{R}^{I_1\times \cdots\times I_M\times  J_1 \times \cdots\times J_N},~~\mc{M}\in\mathbb{R}^{I_1 \times \cdots\times I_M\times I_1 \times \cdots\times I_M},$\\ $\mc{N}\in\mathbb{R}^{J_1 \times \cdots\times J_N\times  J_1 \times \cdots\times J_N}$ be three Boolean tensors with $\mathfrak{R}(\mc{A})=\mathfrak{R}(\mc{A}\n\mc{N})$ and $\mathfrak{R}(\mc{A}^T)=\mathfrak{R}(\mc{A}^T\m\mc{M}^T).$  If $\mc{A}_{\mc{M},\mc{N}}^{\dagger}$ exists, then
\begin{enumerate}
\item[(a)] $\mc{A}\n\mc{N}^T\n\mc{A}^T = \mc{A}\n\mc{N}\n\mc{A}^T;$
\item[(b)] $\mc{A}^T\m\mc{M}^T\m\mc{A} = \mc{A}^T\m\mc{M}\m\mc{A}; $
\item[(c)] $ \mc{A}_{\mc{M},\mc{N}}^{\dagger} $  is unique.
\end{enumerate}
\end{theorem}
 \begin{proof}
 Let $\mc{X}$ be a weighted Moore-Penrose inverse of $\mc{A}.$ Now 
 \begin{eqnarray*}
 \mc{A}\n\mc{N}^T\n\mc{A}^T &=&\mc{A}\n {\mc{N}^T\n\mc{A}^T\m\mc{X}^T}\n\mc{A}^T
 =\mc{A}\n {(\mc{X}\m\mc{A}\n\mc{N})^T}\n\mc{A}^T\\
 &=& {\mc{A}\n\mc{X}\m\mc{A}}\n\mc{N}\n\mc{A}^T
 = \mc{A}\n\mc{N}\n\mc{A}^T.
 \end{eqnarray*}
  This completes the proof of part $(a).$ Using the similar lines of part $(a)$ and relation $(3)$ of Definition \ref{wmpi}, we can  prove part $(b).$ Next we will claim the uniqueness of $A^\dagger_{\mc{M},\mc{N}}.$ \\
 Suppose there exists two weighted Moore-Penrose inverses (say $\mc{X}_1$ and $\mc{X}_2$)  for $\mc{A}.$ Then  
 \begin{eqnarray*}
 \mc{X}_1\m\mc{A}\n\mc{N} &=&  \mc{X}_1\n\mc{A}\n {\mc{X}_2\m\mc{A}\n\mc{N}}
   = \mc{X}_1\m {\mc{A}\n\mc{N}^T\n\mc{A}^T}\m\mc{X}_2^T\\
   &= & {\mc{X}_1\m\mc{A}\n\mc{N}}\n\mc{A}^T\m\mc{X}_2^T 
   = \mc{N}^T\n {\mc{A}^T\m\mc{X}_1^T\n\mc{A}^T}\m\mc{X}_2^T\\
  & =& \mc{N}^T\n\mc{A}^T\m\mc{X}_2^T
   = \mc{X}_2\m\mc{A}\n\mc{N}.
   \end{eqnarray*}
   Since $\mathfrak{R}(\mc{A})=\mathfrak{R}(\mc{A}\n\mc{N}).$ Which implies there exists $\mc{U}$ such that $\mc{A}\n\mc{N}\n\mc{U} = \mc{A}.$ Thus $\mc{X}_1\m\mc{A}\n\mc{N}\n\mc{U} = \mc{X}_2\m\mc{A}\n\mc{N}\n\mc{U}.$ Hence $\mc{X}_1\m\mc{A} =\mc{X}_2\m\mc{A}$. Therefore
   \begin{equation}\label{eq4.51}
      \mc{X}_1=\mc{X}_1\m\mc{A}\n\mc{X}_1=\mc{X}_2\m\mc{A}\n\mc{X}_1.
   \end{equation}
 Now by using Eq. (\ref{eq4.51}), we get
 \begin{eqnarray*}
   \mc{M}\m\mc{A}\n {\mc{X}_1} &=&   {\mc{M}\m\mc{A}\n\mc{X}_2}\m\mc{A}\n\mc{X}_1
   = \mc{X}_2^T\n {\mc{A}^T\m\mc{M}^T\m\mc{A}}\n\mc{X}_1\\
   &=& \mc{X}_2\n\mc{A}^T\m {\mc{M}\m\mc{A}\n\mc{X}_1}
   = \mc{X}_2^T\n {\mc{A}^T\m\mc{X}_1^T\n\mc{A}^T}\m\mc{M}^T\\
   &=& \mc{X}_2^T\n\mc{A}^T\n\mc{M}^T
  = \mc{M}\m\mc{A}\n\mc{X}_2.
   \end{eqnarray*}
  Again as $\mathfrak{R}(\mc{A}^T)=\mathfrak{R}(\mc{A}^T\m\mc{M}^T).$ This implies there exists $\mc{V}^T$ such that $\mc{A}^T\m\mc{M}^T\m\mc{V}^T = \mc{A}^T.$ It leads $\mc{V}\m\mc{M}\m\mc{A}=\mc{A}.$ Thus $ \mc{V}\m\mc{M}\m\mc{A}\n\mc{X}_1 = \mc{V}\m\mc{M}\m\mc{A}\n\mc{X}_2.$  Hence $\mc{A}\n\mc{X}_1 =\mc{A}\n\mc{X}_2$. Therefore
   \begin{equation}\label{eq4.52}
      \mc{X}_2=\mc{X}_2\m\mc{A}\n\mc{X}_2=\mc{X}_2\m\mc{A}\n\mc{X}_1.
   \end{equation}
  Combining Eq. (\ref{eq4.51}) and (\ref{eq4.52}), we obtain $\mc{X}_1=\mc{X}_2$ and hence the proof is complete. 
  \end{proof}
 The existence of weighted Moore-Penrose inverse is not trivial like other generalized inverses. The next theorem discusses the existence of weighted Moore-Penrose inverse. 
 
 \begin{theorem}\label{ewmpi}
 Let $\mc{A}\in\mathbb{R}^{I_1\times \cdots\times I_M\times  J_1\times \cdots\times J_N},~~\mc{M}\in\mathbb{R}^{I_1\times \cdots \times I_M\times I_1\times \cdots\times I_M},$\\ $\mc{N}\in\mathbb{R}^{J_1 \times \cdots\times  J_N\times  J_1 \times \cdots\times J_N}$ be three Boolean tensors with $\mathfrak{R}(\mc{A})=\mathfrak{R}(\mc{A}\n\mc{N})$ and $\mathfrak{R}(\mc{A}^T)=\mathfrak{R}(\mc{A}^T\m\mc{M}^T).$ If
 $ \mc{M}\geq\mc{I}$ and $\mc{N}\geq\mc{I},$ then $\mc{A}_{\mc{M},\mc{N}}^\dagger$ exists if and only if any one of the following conditions holds:
 \begin{enumerate}
 \item[(a)] $\mc{A}\n\mc{N}\n\mc{A}^T\m\mc{M}\m\mc{A} = \mc{A}.$
 \item[(b)] $\mc{A}\n\mc{N}^T\n\mc{A}^T\m\mc{M}\m\mc{A} = \mc{A}.$
 \item[(c)] $\mc{A}\n\mc{N}\n\mc{A}^T\m\mc{M}^T\m\mc{A} = \mc{A}.$
 \item[(d)] $\mc{A}\n\mc{N}^T\n\mc{A}^T\m\mc{M}^T\m\mc{A} = \mc{A}.$
 \end{enumerate}
 In particular, $\mc{A}_{\mc{M},\mc{N}}^\dagger = \mc{N}^T\n\mc{A}^T\m\mc{M}^T.$
 \end{theorem}
  \begin{proof}
Assume $\mc{A}_{\mc{M},\mc{N}}^\dagger$ exists and let $\mc{X}=\mc{A}_{\mc{M},\mc{N}}^\dagger$. Let $\mc{B}= \mc{A}^T\m\mc{A}$. Since for every Boolean tensor, there are finitely many  Boolean tensors of same order, so there must exist positive integers $s,t\in\mathbb{N}$ such that 
\begin{equation}\label{eq4.6}
    (\mc{A}\n\mc{N}\n\mc{A}^T\m\mc{M}^T)^s =  (\mc{A}\n\mc{N}\n\mc{A}^T\m\mc{M}^T)^{s+t}.
\end{equation}
Without loss of generality, we can assume that $s$ is the smallest positive integer for which Eq. (\ref{eq4.6}) holds. Now we will claim $s=1.$ Suppose on contradiction, assume $s>1.$ Now using Eq. (\ref{eq4.6}), and properties of weighted Mooore-Penrose inverse, we get 
\begin{eqnarray}\label{eq4.7}
\nonumber
 {\mc{X}\m \mc{A}\n\mc{N}}\n\mc{A}^T\m\mc{M}^T\m&&\hspace*{-0.7cm}(\mc{A}\n\mc{N}\n\mc{A}^T\m\mc{M}^T)^{s-1} .\\
\nonumber
&&\hspace*{-3.5cm}=   {\mc{X}\m\mc{A}\n\mc{N}}\n\mc{A}^T\m\mc{M}^T\m(\mc{A}\n\mc{N}\n\mc{A}^T\m\mc{M}^T)^{s-1+t}\\
\nonumber\textnormal{This yield~~} \mc{N}^T\n\mc{A}^T\m\mc{X}^T\n\mc{A}^T\m\mc{M}^T&&\hspace*{-0.7cm}(\mc{A}\n\mc{N}\n\mc{A}^T\m\mc{M}^T)^{s-1} \\
&&\hspace*{-4.5cm}=\mc{N}^T\n\mc{A}^T\n\mc{X}^T\m\mc{A}^T\m\mc{M}^T(\mc{A}\n\mc{N}\n\mc{A}^T\m\mc{M}^T)^{s-1+t}.
\end{eqnarray}
Since $\mathfrak{R}(\mc{A})=\mathfrak{R}(\mc{A}\n\mc{N}),$ which implies there exists a  tensor $\mc{U}$ such that $\mc{A}\n\mc{N}\n\mc{U}=\mc{A}.$ Now premultiplying $\mc{U}^T$ to Eq. (\ref{eq4.7}) and using the properties $\mc{U}^T\n\mc{N}^T\n\mc{A}^T=\mc{A}^T$ and $\mc{A}^T \m\mc{X}^T\n\mc{A}^T=\mc{A}^T,$ we get
\begin{equation}\label{eq4.71}
    \mc{A}^T\m\mc{M}^T\m(\mc{A}\n\mc{N}\n\mc{A}^T\m\mc{M}^T)^{s-1} = \mc{A}^T\m\mc{M}^T\m(\mc{A}\n\mc{N}\n\mc{A}^T\m\mc{M}^T)^{s-1+t}.
\end{equation}
 Again, premultiplying $\mc{X}^T $ to Eq. (\ref{eq4.71}) and using the symmetricity of  $\mc{M}\m\mc{A}\n\mc{X}$, we get
$\mc{M}\m {\mc{A}\n\mc{X}\m(\mc{A}}\n\mc{N}\n\mc{A}^T\m\mc{M}^T)^{s-1}= \mc{M}\m {\mc{A}\n\mc{X}\m(\mc{A}}\n\mc{N}\n\mc{A}^T\m\mc{M}^T)^{s-1+t}.$ This gives 
\begin{equation}\label{eq4.8}
    \mc{M}\m(\mc{A}\n\mc{N}\n\mc{A}^T\m\mc{M}^T)^{s-1} = \mc{M}\m(\mc{A}\n\mc{N}\n\mc{A}^T\m\mc{M}^T)^{s-1+t}.
\end{equation}       
 Since $\mathfrak{R}(\mc{A}^T)=\mathfrak{R}(\mc{A}^T\n\mc{M}^T),$ which implies there exists a  tensor $\mc{Z}$ such that $\mc{Z}\m\mc{M}\m\mc{A}=\mc{A}.$ Premultiplying $\mc{Z}$ to Eq. (\ref{eq4.8}) yields 
  $$(\mc{A}\n\mc{N}\n\mc{A}^T\m\mc{M}^T)^{s-1} = (\mc{A}\n\mc{N}\n\mc{A}^T\m\mc{M}^T)^{s-1+t}.$$
  and contradicts the minimality of $s.$ Therefore 
  \begin{equation}\label{eq4.9}
      \mc{A}\n\mc{N}\n\mc{A}^T\m\mc{M}^T = (\mc{A}\n\mc{N}\n\mc{A}^T\m\mc{M}^T)^{t+1},~\mbox{for some }~t\in\mathbb{N}.
  \end{equation}
 Premultiplying Eq. (\ref{eq4.9}) by $\mc{X},$ and using $(\mc{X}\m\mc{A}\n\mc{N})^T=\mc{X}\m\mc{A}\n\mc{N},$ we obtain 
$$
\mc{N}^T\n {\mc{A}^T\m\mc{X}^T\n\mc{A}^T}\m\mc{M}^T = \mc{N}^T\m {\mc{A}^T\n\mc{X}^T\n\mc{A}^T}\m\mc{M}^T\m(\mc{A}\n\mc{N}\n\mc{A}^T\m\mc{M}^T)^t.$$
Since $ \mc{A}^T\m\mc{X}^T\n\mc{A}^T=\mc{A}^T,$ we get
\begin{equation}\label{eqn121}
   \mc{N}^T\n\mc{A}^T\m\mc{M}^T = \mc{N}^T\n\mc{A}^T\m\mc{M}^T\m(\mc{A}\n\mc{N}\n\mc{A}^T\m\mc{M}^T)^t. 
\end{equation}
Premultiplying Eq. (\ref{eqn121}) by a tensor $\mc{U}^T$ and using $\mathfrak{R}(\mc{A})=\mathfrak{R}(\mc{A}\n\mc{N}),$ we again obtain $ \mc{A}^T\m\mc{M}^T = \mc{A}^T\m\mc{M}^T\m(\mc{A}\n\mc{N}\n\mc{A}^T\m\mc{M}^T)^t.$ Postmultiplying $\mc{Z}^T$ and applying $\mathfrak{R}(\mc{A}^T)=\mathfrak{R}(\mc{A}^T\m\mc{M}^T),$  we have  
$$
 \mc{A}^T = \mc{A}^T\m\mc{M}^T\m(\mc{A}\n\mc{N}\n\mc{A}^T\m\mc{M}^T)^{t-1}\m \mc{A}\n\mc{N}\n\mc{A}^T.$$

Now 
\begin{eqnarray*}
 \mc{A}^T &=& \mc{A}^T\m {\mc{M}^T\m(\mc{A}\n\mc{N}\n\mc{A}^T}\m\mc{M}^T)^{t-1}\m\mc{A}\n\mc{N}\n\mc{A}^T\\
 &=&\mc{A}^T\m(\mc{M}^T\m\mc{A}\n\mc{N}\n\mc{A}^T)\m {\mc{M}^T\m(\mc{A}\n\mc{N}\n\mc{A}^T}\m\mc{M}^T)^{t-2}\n\mc{A}\n\mc{N}\n\mc{A}^T\\
 &=&\mc{A}^T\m(\mc{M}^T\m\mc{A}\n\mc{N}\n\mc{A}^T)^2\m {\mc{M}^T\m(\mc{A}\n\mc{N}\n\mc{A}^T}\m\mc{M}^T)^{t-3}\n\mc{A}\n\mc{N}\n\mc{A}^T\\
 &=&\cdots~~~~~~~~~~~~~~\cdots~~~~~~~~~~~~~\cdots\\
 &=&\mc{A}^T\m(\mc{M}^T\m\mc{A}\n\mc{N}\n\mc{A}^T)^{t-2}\m {\mc{M}^T\m(\mc{A}\n\mc{N}\n\mc{A}^T}\m {\mc{M}^T)\n\mc{A}\n\mc{N}\n\mc{A}^T}\\
 &=&\mc{A}^T\m(\mc{M}^T\m\mc{A}\n\mc{N}\n\mc{A}^T)^{t}=\mc{A}^T\m\left[(\mc{A}\n\mc{N}^T\n\mc{A}^T\m\mc{M})^{t}\right]^T.
\end{eqnarray*}
Thus 
\begin{equation}\label{eq4.10}
    \mc{A}=(\mc{A}\n\mc{N}^T\n\mc{A}^T\m\mc{M})^{t}\m\mc{A}.
\end{equation}
As $\mc{M}\geq\mc{I},\mc{N}\geq\mc{I} $, so by Lemma \ref{lemma1}
\begin{equation}\label{eq4.11}
 \mc{A}\n\mc{N}^T\n\mc{A}^T\m\mc{M}\m\mc{A}\geq \mc{A}\n\mc{A}^T\m\mc{A}\geq \mc{A}.
\end{equation}
Postmultiplying $\mc{N}^T\n\mc{A}^T\m\mc{M}\m\mc{A}$, we obtain 
 \begin{equation}\label{eq4.12}
   \mc{A}\n\mc{N}^T\n\mc{A}^T\m\mc{M}\leq(\mc{A}\n\mc{N}^T\n\mc{A}^T\m\mc{M})^2\m\mc{A}.
 \end{equation}
 Combining Eqs.(\ref{eq4.10}), (\ref{eq4.11}) and (\ref{eq4.12}), we have

\begin{eqnarray*}
\mc{A}&\leq&\mc{A}\n\mc{N}^T \n \mc{A}^T\m\mc{M}\m\mc{A}
\leq  (\mc{A}\n\mc{N}^T\n\mc{A}^T*_M\mc{M})^2\m\mc{A} \\
&\leq & (\mc{A}\n\mc{N}^T\n\mc{A}^T\m\mc{M})^3\m\mc{A}\leq \cdots\leq (\mc{A}\n\mc{N}^T\n\mc{A}^T\m\mc{M})^t\m\mc{A}=\mc{A}.
\end{eqnarray*}
Therefore 
\begin{equation}\label{eq4.13}
     \mc{A} = \mc{A}\n\mc{N}^T\n\mc{A}^T\m\mc{M}\m\mc{A},
\end{equation}
and hence completes the proof of the condition $(b).$ By using Theorem \ref{uwmpi}, the other conditions are holds since
\begin{eqnarray}\nonumber
   \mc{A} &=&  {\mc{A}\n\mc{N}^T\n\mc{A}^T}\m\mc{M}\m\mc{A} = \mc{A}\n\mc{N}\n {\mc{A}^T\m\mc{M}\m\mc{A}}\\\label{eq199} &=& {\mc{A}\n\mc{N}\n\mc{A}^T}\m\mc{M}^T\m\mc{A} = \mc{A}\n\mc{N}^T\n\mc{A}^T\m\mc{M}^T\m\mc{A}.
\end{eqnarray}
Further, we will claim not only the four conditions holds but also $\mc{A}^\dagger_{\mc{M},\mc{N}}=\mc{N}^T\n\mc{A}^T\m\mc{M}^T.$ Let $\mc{X}=\mc{A}_{\mc{M},\mc{N}} ^\dagger.$ From Eq. (\ref{eq199}), $\mc{A}=\mc{A}\n\mc{X}\m\mc{A}$ and
\begin{equation*}
    \mc{X}\m\mc{A}\n\mc{X}=\mc{N}^T\n {\mc{A}^T\m\mc{M}^T\m\mc{A}\n\mc{N}^T\n\mc{A}^T}\m\mc{M}^T=\mc{N}^T\n\mc{A}^T\m\mc{M}^T=\mc{X}.
\end{equation*}
Using Theorem \ref{uwmpi}, we show 
\begin{eqnarray*}
\mc{M}\m\mc{A}\n\mc{X} &=&\mc{M}\m {\mc{A}\n\mc{N}^T\n\mc{A}^T}\m\mc{M}^T    =\mc{M}\m\mc{A}\n\mc{N}\n\mc{A}^T\m\mc{M}^T\\
     &=&(\mc{M}\m\mc{A}\n\mc{N}^T\n\mc{A}^T\m\mc{M}^T)^T
=(\mc{M}\m\mc{A}\n\mc{X})^T.
\end{eqnarray*}
Therefore, $\mc{M}\m\mc{A}\n\mc{X}$ is symmetric. Similarly, we can show $\mc{X}\m\mc{A}\n\mc{N}$ is symmetric. So $\mc{A}_{\mc{M},\mc{N}} ^\dagger=\mc{N}^T\n\mc{A}^T\m\mc{M}^T $. Next we will show the converse part. Let $\mc{A}\n\mc{N}\n\mc{A}^T\m\mc{M}\m\mc{A} = \mc{A}.$ Since $\mc{M}\geq \mc{I}$ and $\mc{N}\geq\mc{I}$, so by Lemma \ref{lemma1}, \begin{equation*}
    \mc{A}\leq\mc{A}\n\mc{A}^T\m\mc{A}\leq\mc{A}\n\mc{N}\n\mc{A}^T\m\mc{A}\leq\mc{A}\n\mc{N}\n\mc{A}^T\m\mc{M}\m\mc{A} = \mc{A}
\end{equation*}
 and hence 
\begin{equation}\label{eq4.14}
  \mc{A} =\mc{A}\n\mc{A}^T\m\mc{A} = \mc{A}\n\mc{N}\n\mc{A}^T\m\mc{A}.
\end{equation}
Using the Eq. (\ref{eq4.14}) and symmetricity of $ \mc{A}\n\mc{A}^T$, we obtain
\begin{equation}\label{eq4.15}
   \mc{A}\n\mc{A}^T = \mc{A}\n\mc{N}\n\mc{A}^T\m\mc{A}\n\mc{A}^T
=\mc{A}\n\mc{N}\n\mc{A}^T
=\mc{A}\n\mc{N}^T\n\mc{A}^T .
\end{equation}
Similar argument yields,
\begin{eqnarray}\label{eq4.16}
\nonumber
    \mc{A}^T\m\mc{A}&=&\mc{A}^T\m\mc{M}^T\m {\mc{A}\n\mc{N}^T\n\mc{A}^T}\m \mc{A}=\mc{A}^T\m\mc{M}^T\m {\mc{A}\n\mc{N}\n\mc{A}^T\m\mc{A}}\\&=&\mc{A}^T\m\mc{M}\n\mc{A} = \mc{A}^T\m\mc{M}^T\m\mc{A}.
\end{eqnarray}
Using Eqs. (\ref{eq4.14})-(\ref{eq4.16}), it can be easily verified that $\mc{X}=\mc{N}^T\n\mc{A}^T \m\mc{M}^T$ is satisfies all four conditions of the weighted Moore-Penrose inverse. Similarly, one can start from other conditions to verify the same. Thus the proof is complete.
  \end{proof}
  \begin{remark}
 The equality  condition  in Theorem \ref{ewmpi} $(a)$ can be replaced by ${\bf{`\geq'}}.$
 \end{remark}

\subsection{Space Decomposition}
Using the theory of Einstein product, we introduce the definition of the space decomposition for Boolean tensors, which generalizes the matrix space decomposition \cite{rao}.

\begin{definition}\label{FRD}
 Let  $\mc{F}\in\mathbb{R}^{I_1 \times \cdots\times I_M\times K_1 \times \cdots\times K_L}$ and $\mc{R}\in\mathbb{R}^{K_1 \times \cdots\times K_L\times J_1 \times \cdots\times J_N}$ be two tensors with
 \vspace{-.5cm}
\begin{eqnarray*}
     &&(a)~\mc{A} = \mc{F}\kl\mc{R};\\
     &&(b)~\mathfrak{R}(\mc{A}) = \mathfrak{R}(\mc{F});\\
     &&(c)~\mathfrak{R}(\mc{A}^T) = \mathfrak{R}(\mc{R}^T),
\end{eqnarray*}
then the tensor $\mc{A}$ is called space decomposable and this decomposition is called a space decomposition  of $\mc{A}$.
\end{definition}

In connection with the fact of the  above Definition \ref{FRD} and Lemma \ref{range-stan}, one can conclude the existence of a generalized inverse, as follows.

\begin{theorem}\label{exisgen}
Let $\mc{A}=\mc{F}\kl\mc{R}$ be a space decomposition of  $\mc{A}\in\mathbb{R}^{I_1 \times \cdots\times I_M\times  J_1 \times \cdots\times J_N},$ where $\mc{F}\in\mathbb{R}^{I_1 \times \cdots\times I_M\times K_1 \times \cdots\times K_L}$ and $\mc{R}\in\mathbb{R}^{K_1 \times \cdots\times K_L\times J_1 \times \cdots\times J_N}$ .   Then $\mc{A}^{(1)}$ exists.
\end{theorem}

We now present one of our essential result which represents not only the existence of reflexive generalized inverse but also other inverses through this decomposition.

\begin{theorem}\label{the336}
Let $\mc{X}$ be a generalized inverse of the Boolean tensor $\mc{A}.$ If $\mc{A}=\mc{F}\kl\mc{R}$ is a space decomposition of $\mc{A},$ where $\mc{F}\in\mathbb{R}^{I_1 \times \cdots\times I_M\times K_1 \times \cdots\times K_L}$ and $\mc{R}\in\mathbb{R}^{K_1 \times \cdots\times K_L\times J_1 \times \cdots\times J_N}.$  Then the following are holds:
\begin{enumerate}
\label{eqvspace}
    \item[(a)] $\mc{F}^{(1)}$ and $\mc{R}^{(1)}$ exists.
    \item[(b)] $\mc{F}^{(1)}\m\mc{F}=\mc{R}\n\mc{R}^{(1)}.$
    \item[(c)] $\mc{F}^{(1)}\m\mc{A}=\mc{R}$ and $\mc{A}\n\mc{R}^{(1)}=\mc{F}.$
    \item[(d)] $\mc{R}^{(1)}\m\mc{F}^{(1)}$ is a generalized inverse of $\mc{A}.$
    \item[(e)] $\mc{R}\n\mc{X}$ is a reflexive inverse of $\mc{F}$ and $\mc{X}\m\mc{F}$ is a reflexive inverse of $\mc{R}.$ 
\end{enumerate}
\end{theorem}
\begin{proof}
Since $\mc{X}$ is the generalized inverse of $\mc{A}.$ Then we have  $\mc{A}\n\mc{X}\m\mc{A}=\mc{A},$ which implies\\ 
$\mc{F}\kl\mc{R}\n\mc{X}\m\mc{F}\kl\mc{R}=\mc{F}\kl\mc{R}=\mc{I}\m\mc{F}\kl\mc{R}.$ Further, using Corollary \ref{rtcan}, we get $\mc{F}\kl\mc{R}\n\mc{X}\m\mc{F}=\mc{F}.$ Thus $\mc{R}\n\mc{X}$ is a generalized inverse of $\mc{F}.$ Similarly, one can determine $\mc{X}\m\mc{F}$ is a generalized inverse of $\mc{R}$.  Hence $(a)$ is proved. Now using the result $(a),$ one can prove $(b)$ and $(c).$  To prove $(d)$ we use the fact $(a)$ and obtain. 
\begin{equation*}
    \mc{A}\n\mc{R}^{(1)}\kl\mc{F}^{(1)}\m\mc{A}=  \mc{A}\n\mc{X}\m\mc{F}\kl\mc{R}\n\mc{X}\m\mc{A}=  \mc{A}\n\mc{X}\m\mc{A}\n\mc{X}\m\mc{A}=\mc{A}.
\end{equation*}
Hence $\mc{R}^{(1)}\m\mc{F}^{(1)}$ is a generalized inverse of $\mc{A}.$ 
In a similar manner, one can prove $(e)$ using the fact $\mc{R}\n\mc{X}\m\mc{F}\kl\mc{R}\n\mc{X}
=\mc{R}\n\mc{X}$ and $\mc{X}\m\mc{F}\kl\mc{R}\n\mc{X}\m\mc{F}
=\mc{X}\m\mc{F}.$ This completes the proof.
\end{proof}
In view of the above theorem one can draw a conclusion, as follows. 
\begin{remark}\label{rmk3.43}
 Every generalized inverse of $\mc{A}$ need not of the form $\mc{R}^{(1)}\kl\mc{F}^{(1)}.$ 
 \end{remark}
 We verify the Remark \ref{rmk3.43} with the following example.
\begin{example}\label{example3.45}
Let
$~\mc{A}=(a_{ijkl}) \in \mathbb{R}^{{2\times3}\times{2 \rtimes 3}}$ be a Boolean tensor with
\begin{eqnarray*}
a_{ij11} =
    \begin{pmatrix}
    1 & 1 & 0 \\
    1 & 0 &  0
    \end{pmatrix},
a_{ij12}=a_{ij13} =a_{ij21}=a_{ij22}=a_{ij23}=
    \begin{pmatrix}
     0 & 0 & 0\\
     0 & 0 & 0
    \end{pmatrix}.
\end{eqnarray*}
Consider  $\mc{A}^{(1)}=(x_{ijkl}) \in \mathbb{R}^{{2\times3}\times{2 \rtimes 3}}$ is a generalized inverse of $\mc{A}$ with 
\begin{eqnarray*}
x_{ij11} =
    \begin{pmatrix}
    1 & 0 & 0 \\
    0 & 0 &  0
    \end{pmatrix},
x_{ij12} =
    \begin{pmatrix}
     0 & 1 & 0\\
     0 & 0 & 0
    \end{pmatrix},
x_{ij13} =
    \begin{pmatrix}
     0 & 0 & 0\\
     0 & 0 & 0
\end{pmatrix},\\
x_{ij21} =
    \begin{pmatrix}
     0 & 0 & 0\\
     1 & 0 & 0
    \end{pmatrix},
    x_{ij22} =
    \begin{pmatrix}
    0 & 0 & 0\\
    0 & 0 &  0
    \end{pmatrix},
x_{ij23} =
    \begin{pmatrix}
     0 & 0 & 0\\
     0 & 0 & 0
    \end{pmatrix}.
\end{eqnarray*}

In light of the Theorem \ref{the336} (e) one can conclude
$$\mc{R}^{(1)}*_2\mc{F}^{(1)}=\mc{A}^{(1)}*_2\mc{F}*_2\mc{R}*_2\mc{A}^{(1)}=\mc{A}^{(1)}*_2\mc{A}*_2\mc{A}^{(1)} \neq \mc{A}^{(1)}.$$
Therefore, every generalized inverse of $\mc{A}$ need not of the form $\mc{R}^{(1)}\kl\mc{F}^{(1)}.$ 
\end{example}
At this point one may be interested to know when does the generalized inverse of a Boolean tensor of the form $\mc{R}^{(1)}\kl\mc{F}^{(1)}$ ?  The answer to this question is explained in the following Remark. 
\begin{remark}\label{rmk3.46}
 If $\mc{X}$ is a reflexive inverse of a Boolean tensor $\mc{A}\in \mathbb{R}^{I_1 \times \cdots\times I_M\times J_1 \times \cdots\times J_N}$ and $\mc{A}=\mc{F}\kl\mc{R}$ has a space decomposition, where $\mc{F}\in\mathbb{R}^{I_1 \times \cdots\times I_M\times K_1 \times \cdots\times K_L}$ and $\mc{R}\in\mathbb{R}^{K_1 \times \cdots\times K_L\times J_1 \times \cdots\times J_N}.$ Then every generalized inverse is of the form $\mc{R}^{(1)}\kl\mc{F}^{(1)}.$
\end{remark}
Now considering the fact of Remark \ref{rmk3.46} and the observation of the Example \ref{example3.45},  one can get the desired result.

\begin{theorem}\label{refeqv}
Let  $\mc{A} \in \mathbb{R}^{I_1 \times \cdots\times I_M\times J_1 \times \cdots\times J_N}$ be a Boolean tensor and 
$\mc{A}=\mc{F}\kl\mc{R}$ be a space decomposition of $\mc{A}, $ where $\mc{F}\in\mathbb{R}^{I_1 \times \cdots\times I_M\times K_1 \times \cdots\times K_L},$  $\mc{R}\in\mathbb{R}^{K_1 \times \cdots\times K_L\times J_1 \times \cdots\times J_N}.$ Assume that generalized inverse of  either $\mc{F}$ reflexive or $\mc{R}$ reflexive. Then $\mc{X}$  is a reflexive generalized inverse of $\mc{A}$  if and only if $\mc{X} = \mc{R}^{(1)}\kl\mc{F}^{(1)}.$
\end{theorem}
\begin{proof}
Consider the generalized inverse of $\mc{F}$ is reflexive. Taking into account of Theorem \ref{eqvspace} $(d)$, we obtain  $\mc{X} = \mc{R}^{(1)}\kl\mc{F}^{(1)}$, which is a generalized inverse of  $\mc{A}$. Therefore, it is enough to show $\mc{X}\m\mc{A}\n\mc{X}=\mc{X}.$ Now using Theorem \ref{eqvspace} $(c)$, we get
\begin{eqnarray*}
\mc{X}\m\mc{A}\n\mc{X} =\mc{R}^{(1)}\kl\mc{F}^{(1)}\m {\mc{A}\n\mc{R}^{(1)}}\kl\mc{F}^{(1)}
=\mc{R}^{(1)}\kl {\mc{F}^{(1)}\m\mc{F}\kl\mc{F}^{(1)}} = \mc{R}^{(1)}\kl\mc{F}^{(1)}.
\end{eqnarray*}
Conversely, let $\mc{X}$ be a reflexive inverse of $\mc{A}.$ Then by Theorem \ref{eqvspace} $(e)$,
\begin{eqnarray*}
\mc{X} = \mc{X}\m\mc{A}\n\mc{X} = \mc{X}\m\mc{F}\kl\mc{R}\n\mc{X} = \mc{R}^{(1,2)}\kl\mc{F}^{(1,2)}=\mc{R}^{(1)}\kl\mc{F}^{(1)}.
\end{eqnarray*}
\vspace{-.3cm}
\end{proof}


\begin{remark}\label{rk2.41}
If we drop the condition either $\mc{F}$ or $\mc{R}$ is reflexive generalized inverse of $\mc{A}$ in Theorem \ref{refeqv}, then the theorem will not true in general.
 \end{remark}

In favour of the the Remark \ref{rk2.41} we produce an example as follows.

\begin{example}
Let $\mc{A}$ be the Boolean tensor defined in Example \ref{example3.45} and $\mc{A}=\mc{F}=\mc{R}.$ Since $\mc{A}\2\mc{I}\2\mc{A}=\mc{I}$ and $\mc{I}\2\mc{A}\2\mc{I}\neq\mc{I}$, it follows that $\mc{I}$ is the generalized inverse for both $\mc{F}$ and $\mc{G}$ but not reflexive. In view of the Theorem \ref{refeqv},   
one can conclude 
$\mc{R}^{(1)}\2\mc{F}^{(1)}=\mc{I}$
is not a reflexive generalized inverse of $\mc{A}.$
\end{example}

In \cite{beasley} and \cite{song}, the authors have defined the rank of a Boolean matrix through space decomposition. Next, we discuss the rank and weight of a Boolean tensors.

\begin{definition}
 Let $\mc{A}\in\mathbb{R}^{I_1\times\cdots\times I_M\times J_1 \times\cdots\times
 J_N}$ be a Boolean tensor. If there exist a least positive integer,   $r=K_1 \times\cdots\times K_L$  such that the Boolean tensors $\mc{B}\in\mathbb{R}^{I_1 \times\cdots\times I_M\times K_1\times\cdots\times  K_L}$ and $\mc{C}\in\mathbb{R}^{K_1 \times\cdots\times K_L\times J_1\times\cdots\times J_N}$ satisfies  $\mc{A}=\mc{B}\kl\mc{C}$.  Then  $r$ is called the Boolean rank of $\mc{A}$ and denoted by $r_b(\mc{A}).$
\end{definition}

\begin{example}\label{exrank}
Consider a Boolean tensor
$~\mc{A}=(a_{ijkl}) \in \mathbb{R}^{{2\times2}\times{2 \rtimes 2}}$ with entries
\begin{eqnarray*}
a_{ij11} =
    \begin{pmatrix}
    1 & 0  \\
    0 &  0
    \end{pmatrix},~
a_{ij12} =
    \begin{pmatrix}
     0 & 0\\
     0 &  0
    \end{pmatrix},~
a_{ij21} =
    \begin{pmatrix}
     0 & 0\\
     1 &  0
    \end{pmatrix},~
a_{ij22} =
    \begin{pmatrix}
    1 &  0\\
    0 &  0
    \end{pmatrix}.
    \end{eqnarray*}
    There exist a least positive integer $r=2$ and two tensor
    $~\mc{B}=(b_{ijk}) \in \mathbb{R}^{{2\times2 \times 2}}$ and $~\mc{C}=(c_{ijk}) \in \mathbb{R}^{{2\times2 \times 2}}$ with entries
    \begin{eqnarray*}  
    b_{ij1} =
    \begin{pmatrix}
    1 &  0\\
    0 &  0
    \end{pmatrix},
       b_{ij2} =
    \begin{pmatrix}
    0 &  0\\
    1 &  0
    \end{pmatrix},
       c_{ij1} =
    \begin{pmatrix}
    1 &  0\\
    0 &  1
    \end{pmatrix},
       c_{ij2} =
    \begin{pmatrix}
    0 &  1\\
    0 &  0
    \end{pmatrix},
    \end{eqnarray*}
    such that $\mc{A}= \mc{B}*_1\mc{C}$. However, $r=1$ gives two matrices $B$ and $C$, which is impossible to get a tensor.  Thus rank of the tensor is $2$.
\end{example}

On the other hand, the rank of the Boolean tensor is zero if it is zero tensor. Further, we have $\mc{A}=\mc{I}_m\m\mc{A}=\mc{A}\n\mc{I}_n$, where $\mc{A}\in\mathbb{R}^{I_1\times\cdots\times I_M\times J_1 \times\cdots\times
 J_N} $. It is quite apparent that
$$0\leq r_b(\mc{A})\leq \min\{I_1 \times\cdots\times I_M,~J_1\times\cdots\times  J_N\}. $$

To prove the last result of this paper, we define weight of Boolean tensor as.

\begin{definition}
The weight of Boolean tensor is denoted by $w(\mc{A})$ and defined as
$$
w(\mc{A})=\{\mbox{ Total number of non zero elements of } \mc{A}\}.
$$
\end{definition}

The existence of generalized inverse can be discussed through Boolean rank, as follows.

\begin{theorem}\label{rankthm}
Let $\mc{A}\in\mathbb{R}^{I_1\times \cdots \times I_M\times J_1 \times \cdots \times  J_N}$ be any tensor with $r_b(\mc{A})\leq 1.$ Then $\mc{A}$ is regular.
\end{theorem}

\begin{proof}
It is trivial for $r_b(\mc{A})=0,$ as a consequence of the fact $\mc{O}$ tensor is always regular.
 Further, consider $r_b(\mc{A})=1$ and define  a tensor $\mc{J},$ with no zero elements. Then there exist permutation tensors $P$ and $\mc{Q}$ such that $\mc{P}\m\mc{A}\n\mc{Q}=\begin{bmatrix}
\mc{J} & \mc{O}\\
\mc{O} & \mc{O}\\
\end{bmatrix}.$ As $\mc{J}$ is regular, it implies $\mc{P}\m\mc{A}\n\mc{Q}$ is regular.
 In view of the Lemma \ref{block} and Preposition \ref{permu} one can conclude $\mc{A}$ is regular 
\end{proof}
It is clear, if the weight of a Boolean tensor is $1,$ then the rank is also $1$. In view of this we obtain the following result. 
\begin{corollary}\label{weightcoro}
Let $\mc{A}\in\mathbb{R}^{I_1\times \cdots I_M\times J_1\times\cdots J_N}$ be any tensor with $w(\mc{A})\leq 1.$ Then $\mc{A}$ is regular.
\end{corollary}

\section{Conclusion}
In this paper, we have introduced generalized inverses $(\{i\}$-inverses $(i = 1, 2, 3, 4))$  with the Moore-Penrose inverse and weighted Moore-Penrose inverse for Boolean tensors via the Einstein product, which is a generalization of the generalized inverses of Boolean matrices. In addition to this, we have discussed their existence and uniqueness. This paper also provides some characterization through complement and its application to generalized inverses. 
Further, we explored the space decomposition for the Boolean  tensors, at the same time, we have studied rank and the weight for the Boolean tensor. 
In particular, we limited our study for Boolean tensors with $r_b(\mc{A}) \leq 1$ and $w(\mc{A}) \leq 1$. Herewith left as open problems for future studies.\\
{\bf Problem:} If the Boolean rank or weight  of a tensor $\mc{A}$ is greater than 1, then under which conditions the Boolean tensor $\mc{A}$ is regular $?$\\
Additionally, it would be interesting to investigate more generalized inverses on the Boolean tensors; this work is currently underway.

\noindent {\bf{Acknowledgments}}\\
This research work was supported by Science and Engineering Research Board (SERB), Department of Science and Technology, India, under the Grant No. EEQ/2017/000747.

\bibliographystyle{abbrv}
\bibliographystyle{vancouver}
\bibliography{Boolean}
\end{document}